\documentclass[12pt]{article}
\usepackage{amsmath,amssymb,a4wide}
\usepackage{theorem}
\usepackage{calc}
\usepackage{footnote}

\usepackage{epsf}           
\usepackage{graphicx}       
\usepackage{epsfig}         
\usepackage{pstricks}
\usepackage{psfrag} 

\usepackage{algorithm,algorithmic}

\newtheorem{theorem}{Theorem}

\newtheorem{lemma}[theorem]{Lemma}

\newtheorem{proposition}[theorem]{Proposition}
\newtheorem{remark}[theorem]{Remark}

\newenvironment{proof}[1][Proof]{\noindent\textbf{#1.} }{\ \rule{0.5em}{0.5em}}

\theoremstyle{plain}{\theorembodyfont{\rmfamily}

\newenvironment{deflist}[1][\quad\quad]%
{\begin{list}{}{%
\settowidth{\labelwidth}{\textrm{#1~}}%
\setlength{\leftmargin}{\labelwidth+\labelsep}}}
{\end{list}}

\begin{document}

\title{On discretization in time in simulations of particulate flows}
\author{Matthieu Hillairet, Alexei Lozinski and Marcela Szopos\thanks{
e-mails: matthieu.hillairet@math.univ-toulouse.fr, alexei.lozinski@math.univ-toulouse.fr, szopos@mip.ups-tlse.fr.
}\\
Institut de Math\'{e}matiques de Toulouse, Universit\'{e} de Toulouse, France
}
\date{}
\maketitle

\begin{abstract}
We propose a time discretization scheme for a class of ordinary differential equations arising in simulations of fluid/particle flows.
The scheme is intended to work robustly in the lubrication regime when the distance between two particles immersed in the fluid or between a particle and the wall
tends to zero. The idea consists in introducing a small threshold for the particle-wall distance below which the real trajectory of the particle
is replaced by an approximated one where the distance is kept equal to the threshold value. The error of this approximation is estimated both theoretically and by numerical experiments. 
Our time marching scheme can be easily incorporated into a full simulation method where the velocity of the fluid is obtained by a numerical solution to Stokes or Navier-Stokes equations. 
We also provide a derivation of the asymptotic expansion for the lubrication force (used in our numerical experiments) acting on a disk immersed in a Newtonian fluid and approaching the wall. 
The method of this derivation is new and can be easily adapted to other cases.
\end{abstract}

\section{Introduction}
One of the challenges for fluid/particle flow simulations is to provide an accurate resolution of the lubrication regime 
when the distance between two particles immersed in the fluid or between a particle and the wall becomes very small. Taking aside the problems related to the discretization in space 
(extremly high gradients of the velocity in the narrow gap between the particle and the wall, for example), we focus our attention in this article on the discretization in time. The lubrication regime 
is characterized by very high magnitude of the drag force (which can be referred to as the lubrication force in this case). Very small step sizes in time should be thus employed in order to obtain
a physically acceptable solution. Our idea is to prohibit the particle from approaching too closely the wall during a simulation. We shall thus choose a threshold $q_s$ for the distance $q$
between the particle and the wall and replace the true trajectory of the particle by an approximated one, in which the distance $q$ is kept equal to $q_s$ until an eventual rebound of the particle
from the wall. The moment of rebound is predicted using an auxiliary quantity (a crude  approximation of the velocity) that is computed all along the period of time when the particle is stuck
at the distance $q_s$.
This approach reminds the gluey particle model of \cite{Maury2007, Lefebvre09} where $q_s$ is set to zero and the limit of vanishing viscosity is considered. 
However, our motivations are quite different from that behind the gluey particle model. This model is intended as a simple alternative to the standard governing equations of Navier-Stokes type, 
eventually corrected by taking into account the roughness of the particle surface. On the other hand, our approach is to take the standard fluid equations for granted 
(assuming the particle surface to be smooth) and to provide a tool for a robust time discretization of them. It means, in particular, that we would need an accurate enough method 
valid for any given value of the viscosity, not necessarily small. Note also, that our threshold $q_s$ will typically depend on the time step size, so that it is indeed a numerical device 
and it has no physical meaning.

The plan of the article is as follows: we start by reminding the governing equations in a general setting and by explaining in more detail the difficulties related to the
simulations in the lubrication regime in the next section. Section 3 is the core of the paper. The idea of the threshold is rigorously introduced and studied there for the model ordinary differential equation
representing the essence of the general setting in the simplest case of a circular (or spherical) particle approaching the wall. The discussion is held on the continuous level in Section 3.
The discretization in time is introduced in Section 4 where several implementations of our idea are proposed on the discrete level followed by numerical experiments. 
We also include an appendix detailing a derivation of the asymptotic expansion for the lubrication force acting on a disk approaching the wall. 
The method of this derivation is new and can be easily adapted to other cases.

\section{Motivations: governing equations for the fluid/particle flows and some difficulties arising in their simulations}\label{general}

The general setting of this work is a study of the motion of a rigid particle immersed in a viscous, incompressible fluid with a particular emphasis on situations when
the particle approaches the plane. To set the notations, we assume in general that the fluid (with the particle inside) fills a fixed domain $\Omega\in\mathbb{R}^d$
with $d=2$ or $3$, while the region occupied by the particle $\mathcal B_t\subset\Omega$ varies with time $t$. We denote the time-dependent fluid domain $\mathcal F_t$ so that 
$ {\Omega= \overline{\mathcal B_t}\cup
\mathcal F_t}$ at any time $t$.  Supposing that the inertial effects are negligible in the fluid and the no-slip conditions are valid on the boundaries of $\Omega$
and $\mathcal{B}_t$, the fluid motion is governed by the Stokes equations
\begin{figure}[h]
 \psfrag{R}{$R$}
 \psfrag{y}{$y$}
 \psfrag{x}{$x$}
 \psfrag{B}{$\mathcal{B}_t$}
 \psfrag{q}{$q$}
 \psfrag{Ve2}{$\mathbf{V}$}
 \psfrag{omega}{$\mathcal{F}_t=\Omega\setminus\mathcal{B}_t$}
     \centerline{
         \includegraphics[scale=0.5]{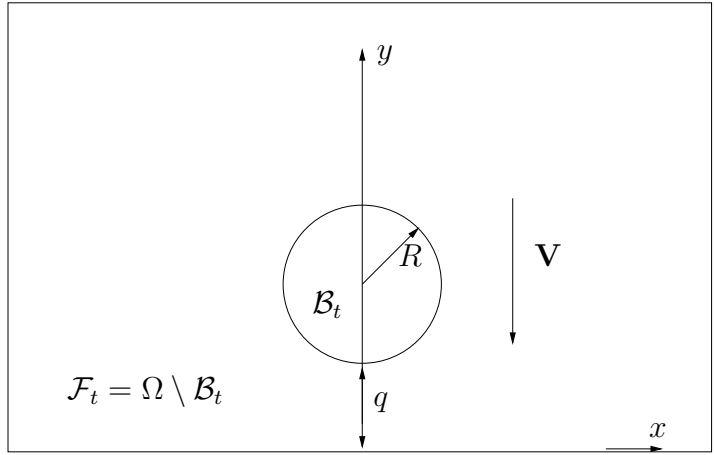}
     }
     \caption{Notations}
     \label{fig_Stokes}
\end{figure}
\begin{equation}\label{stokes}
    \begin{cases}
   -\nu \Delta\mathbf{u} + \nabla p = \rho_f\mathbf{g},\ &\textrm{in}\ \mathcal{F}_t\\
    \nabla\cdot\mathbf{u} = 0, \ &\textrm{in}\ \mathcal{F}_t\\
    \mathbf{u} = \mathbf{0},\ &\textrm{on}\ \partial\Omega\\
    \mathbf{u} = \mathbf{V}+\omega\times\mathbf{r}\ &\textrm{on}\ \partial\mathcal{B}_t
    \end{cases}
\end{equation}
where $\mathbf{u}$ and $p$ are the velocity and the pressure in the fluid, $\nu$ and $\rho_f$ are the viscosity and the density of the fluid, $\mathbf{g}$ is the external force,
$\mathbf{V}=\mathbf{V}(t)$ and $\omega=\omega(t)$ are the translational and angular
velocities of the rigid body $\mathcal{B}_t$, $\mathbf{r} =\mathbf{x}-\mathbf{G}$ is the vector pointing from the center of mass of the particle $\mathbf{G}$ to a point $\mathbf{x}$ on its boundary.
The more realistic Navier-Stokes equations may be accomodated into the framework \eqref{stokes} by including the convective term $\mathbf{u}\cdot\nabla\mathbf{u}$ 
into $\mathbf{g}$.

The fluid exerts a net force $\mathbf{F}$ and a torque $\mathbf{T}$ on the particle given by
\begin{align}
\mathbf{F} &= \mathbf{F}(\mathcal{B}_t,\mathbf{V},\omega) = \int_{\partial \mathcal B_t} (2 \nu D(\mathbf u) - p I) \mathbf n \text{d$\sigma$}, \label{eq_drag}\\
\mathbf{T} &= \mathbf{T}(\mathcal{B}_t,\mathbf{V},\omega) = \int_{\partial \mathcal B_t} \mathbf{r}\times(2 \nu D(\mathbf u) - p I) \mathbf n \text{d$\sigma$}, \notag
\end{align}
where $D(\mathbf u)$ stands for the symmetric gradient of $\mathbf u$ and $\mathbf n$ is the unit normal
vector on $\partial \mathcal B_t$ directed towards the fluid domain.
Note that $\mathbf{F}$ and $ {\mathbf{T}}$ are indeed functions of only the placement of the particle $\mathcal{B}_t$ and its translational and angular velocities since
the velocity $\mathbf{u}$ and the pressure $p$ are uniquely determined in the fluid by the parameters $\mathcal{B}_t$, $\mathbf{V}$ and $\omega$
as the solution to the Stokes equations (\ref{stokes}). Moreover, the dependence of $\mathbf{F}$ and $ {\mathbf{T}}$ on $\mathbf{V}$ and $\omega$ is linear. Using these
notations we write out the equations of motion of the particle as follows
\begin{equation}\label{rigpt}
    \begin{cases}
    m\frac{d\mathbf{V}}{dt} = \mathbf{F}(\mathcal{B}_t,\mathbf{V},\omega) + m\mathbf{g}, \\
    \mathbb{I}_t\frac{d\omega}{dt} + \omega\times\mathbb{I}_t\omega= \mathbf{T}(\mathcal{B}_t,\mathbf{V},\omega), \\
    \end{cases}
\end{equation}
where $m$ is the mass of the particle and $\mathbb{I}_t$ is its inertia tensor, expressed in the fixed Cartesian frame and
thus dependent on time.  Equations (\ref{rigpt}) are coupled with the equations describing the propagation of the particle, { {\it \textit{i.e.}}} $\dot{\mathbf{G}}=\mathbf{V}$
for the center of mass and $\dot{{ {\mathbf{r}_i}}}=\omega\times{ {\mathbf{r}_i}}$, $i=1,\ldots,d$, for the vectors ${ {\mathbf{r}_i}}$ fixed in the particle.
The force $\mathbf{F}$ is a sum of the Archimedes force due to gravity and of the drag force which is purely hydrodynamic, { {\it \textit{i.e.}}} obtained from (\ref{stokes}) by setting
$\mathbf{g}=0$. The particularity of the drag is that it tends very rapidly to $\infty$ when the  particle approaches the wall, thus preventing collisions between them. Indeed, it has been proved in \cite{Hillairet07} (2D case), \cite{HillairetTakahashi09} (3D case), that a smooth rigid body embedded in a viscous fluid cannot touch the wall in finite time. In the regime of very small distance between the particle and the wall, the drag force is also known as the lubrication force
and it is notoriously difficult to take into account in a numerical simulation.

It is noteworthy that these considerations are valid only for smooth surfaces; however, modelling surface roughness of the wall or particle is a much more delicate issue. 
Several experimental (\cite{SmartL1989}, \cite{Feuillebois1993}, \cite{Feuillebois2004}, \cite{Vinogr2006}) and theoretical (\cite{Feuillebois2004}, \cite{Feuillebois2005}) 
works propose in this case to modify the expression for the lubrication drag force by introducing a shift to the distance to the wall, of magnitude strictly lower than the roughness size. 
This physically means that roughness decreases the dissipation in the system, and that the interaction is similar to that between equivalent smooth surfaces, located at some intermediate position, 
between the peaks and the valleys of asperities. 

To simulate the motion of the particle governed by equation described above, one should be able to solve numerically the Stokes system (\ref{stokes}) for any given particle position $\mathcal{B}_t$
and for any given velocities $\mathbf{V}$ and $\omega$. This is a formidable task in itself especially because the position of the particle is not known a priori but changes with time. We do not make
precise the choice of the numerical method for the Stokes system. We just assume  for the moment that the force $\mathbf{F}$ and the torque $\mathbf{T}$ can be computed for any $\mathcal{B}_t$,
$\mathbf{V}$, $\omega$ but this computation is in general very expensive. Our primary goal in this article is to devise an efficient discretization in time of (\ref{rigpt}) using as few as possible
solutions of the Stokes system. The simplest idea is to use the following scheme
\begin{align}
    & m\frac{\mathbf{V}^k-\mathbf{V}^{k-1}}{\Delta t} = \mathbf{F}(\mathcal{B}_{t_{k-1}},\mathbf{V^k},\omega^k) + m\mathbf{g}(t_k), \label{s1}\\
    & \mathbb{I}_{t_{k-1}}\frac{\omega^k-\omega^{k-1}}{\Delta t}  + \omega^{k-1}\times\mathbb{I}_{t_{k-1}}\omega^k = \mathbf{T}(\mathcal{B}_{t_{k-1}},\mathbf{V^k},\omega^k),  \label{s2}\\
    & \frac{\mathbf{G}^k-\mathbf{G}^{k-1}}{\Delta t} = \mathbf{V}^k,  \label{s3}\\
    & \frac{{ {\mathbf{r}_i}}^k-{ {\mathbf{r}_i}}^{k-1}}{\Delta t} = { {\mathbf{r}_i}}^k\times\omega^k, \quad i=1,\ldots,d.  \label{s4}
\end{align}
We have introduced here the uniform grid in time $t_k=k\Delta t$ and have denoted the quantities computed at the time step $t_k$ by the superscript $k$. The idea behind the scheme
(\ref{s1})--(\ref{s4}) is to compute first the velocities by (\ref{s1})--(\ref{s2}) and then to propagate the particle using the last available values of the velocities. The equations
(\ref{s1})--(\ref{s2}) are thus coupled together and also coupled with the solution of the Stokes system on a fixed geometry given by position of the particle $\mathcal{B}_{t_{k-1}}$
on the previous time step $t_{k-1}$. The cost of such a computation is normally essentially the same as that of the Stokes system (\ref{stokes}) with prescribed $\mathbf{V}$ and $\omega$.
We need thus one solution of the Stokes system per time step.
This approach was successfully used, for example, in \cite{LoRo07} in conjunction with a fictitious domain discretization of the Stokes system as in \cite{Glo00}.
However, independently from the discretization in space, one would encounter problems when the particle approaches the wall.  Indeed, in this lubrication regime the force
$\mathbf{F}$ explodes and thus the scheme (\ref{s1})--(\ref{s2}) is no longer valid unless an extremely low value for the time step is used that makes the simulation
prohibitively expensive. A commonly used cure for this problem is to introduce short-range repulsion forces between the particle and the wall, as in \cite{Glo00}, for example.
However, the influence of these (not necessarily realistic) forces on the accuracy of a simulation is not well understood. Another simple idea is just to stop
the particle when it tries to penetrate the wall during a numerical simulation. However, it is then not necessarily clear what criterion should be chosen to decide if the
particle should eventually bounce off the wall and when should it happen.  These questions have a partial answer in the articles \cite{Maury2007, Lefebvre09} on
the gluey particle model. It is shown there that the particle trajectory satisfies an integro-differential equation in the limit of vanishing viscosity, which is easy to
discretize in time using moderate time steps and which predicts the moment of an eventual rebound from the wall. We pursue a similar idea in this article but our aim is
to construct an approximated trajectory of the particle in the lubrication regime that would be accurate enough for any given value of the viscosity, not necessarily
small.

\section{A model ordinary differential equation with lubrication forces}

\subsection{The model}

Let us consider the simplest setting of the problem described in the previous section: assume $\Omega\subset\mathbb{R}^2$ is the half-plane $\{(x,y),\ y>0\}$
and the particle is a { {disk}} of radius $R$. Let moreover $\mathbf{g}(t)=g(t)\mathbf{e}_2$ and { {assume}} the particle is at rest at the initial time.  
{ {The $x$-component of the particle velocity and its angular velocity will then vanish at all time.}} 
The position of the particle is fully determined by its distance $q$ from the bottom, as in Figure
\ref{fig_Stokes}.  The net force $\mathbf{F}$ is the sum of the drag, which is a function of $q$ and $\mathbf{V}$, linear in ${ {\mathbf{V}}}$, and of the Archimedes
force:
$$
{ {\mathbf{F}=-n(q)\mathbf{V} + m_ag(t)\mathbf{e}_2,  \text{ with $m_a = m - \rho_f |\mathcal B_t|,$}}}
$$
where $n(q)$ is the drag coefficient computed by the Stokes equations (\ref{stokes}). Denoting the $y$-component of the velocity by $v$, we are thus led to the
following differential equations
\begin{equation}\label{odem}
    \begin{cases}
    { {m\dot{v} = -n(q)v + m_ag}}, \\
    \dot{q} = v. \\
    \end{cases}
\end{equation}
We are especially interested in the lubrication regime of small $q$. The asymptotic of $n(q)$ when $q\to 0$
is given  by $3\sqrt{2}\pi \nu \left( \frac{R}{q}\right)^{\frac{3}{2}}$ (see Appendix A) in the 2D case. 
After eliminating $v$ from the system \eqref{odem} and going to non-dimensional variables (see Appendix B for the details) we obtain the following equation for $q(t)$
\begin{equation}\label{ode2D}
\ddot{q}=-\varepsilon \frac{\dot{q}}{q^{\frac 32}} + g,
\end{equation}
with $\varepsilon=\frac{3\nu}{\rho_sR^{\frac 32}}\sqrt{  \frac{2\rho_s}{g_{char}(\rho_s-\rho_f)}  }$
where $\rho_s$ is the density of the solid { {disk}} and $g_{char}$ is the characteristic value of $g$.

In the same way, we can consider the analogous three { {dimensional}} problem setting  $\Omega$ to be the half-space $\{(x,y,z),\ z>0\}$
and the particle to be a ball  of radius $R$. The asymptotic expression is well known in this case (cf. Remark \ref{rem3D}) and is given by $6\pi V \nu \frac{R^2}{q}$. 
Performing the same non-dimensionalizations as in the 2D case, we arrive at the equation for $q(t)$ (the distance from the particle to the wall): 
\begin{equation}\label{ode3D}
\ddot{q}=-\varepsilon \frac{\dot{q}}{q} + g,
\end{equation}
with $\varepsilon=\frac{9\nu}{2\rho_sR^{\frac 32}}\sqrt{  \frac{\rho_s}{g_{char}(\rho_s-\rho_f)}  }$.

\begin{table}[h]
\begin{center}
$
\begin{array}{|c|c|c|c|c|}
\hline
\text{liquid} & \rho _{s}/\rho _{f} & R\text{ (mm)} & \varepsilon\text{ in }\eqref{ode2D} & \varepsilon\text{ in }\eqref{ode3D}  \\
\hline\hline
\text{water} & 1.1 & 1 & 0.13 & 0.14\\
& 1.5 & 0.1 & 1.55 & 1.64 \\
\hline\hline
\text{glycerin} & 1.1 & 1 & 152 & 161\\
& 1.5 & 0.1 & 1843 & 1954 \\
\hline
\end{array}%
$
\end{center}
\caption{Typical values of the parameter $\varepsilon$ in \eqref{ode2D} or  in \eqref{ode3D} taking the viscosity and density of either water or glycerin 
and different value of the particle radius $R$ and density $\rho_s$.}
\label{tabeps}
\end{table}

We remind that equations of the type \eqref{ode2D}--\eqref{ode3D} are at the basis of the gluey particle model of \cite{Maury2007, Lefebvre09}. The model consists in fact in considering 
the limit $\varepsilon\to 0$, which is physically the limit of vanishing viscosity. We see, however, that $\varepsilon$ is not necessarily small.

\subsection{Estimates and an approximated solution for the model ODE}

Consider the ordinary differential equation
\begin{equation}
\ddot{q}=-n(q)\dot{q}+g  \label{ode}
\end{equation}%
or, equivalently, the system of first-order equations
\begin{equation}\label{syst}
\begin{cases}
\dot v & = -n(q)v+g \\
\dot q & = v
\end{cases}
\end{equation}
where $n(q)$ is a given differentiable decreasing positive function on $(0,\infty )$ with $%
n(q)=N^{\prime }(q)$ such that $N(q)\rightarrow -\infty $ as $q\rightarrow 0.$ Note that this is valid for $n(q)=\varepsilon/q^{\frac 32}$ or $n(q)=\varepsilon/q$, which give
the asymptotic of the lubrication force in 2D and 3D respectively.
We assume $g \in L^1_{loc}(\mathbb R ^+)$ from now on. Under this hypothesis, Problem (\ref{ode}), completed with appropriate initial conditions, is well-posed, as proved for instance in \cite[Prop. 1.1]{Maury2007} 
for the case $n(q)=\varepsilon/q$. The proof is generalized without problem to any $n(q)$ satisfying the above hypotheses, as mentioned in \cite[Section 5.1]{Maury2007}. 

\begin{proposition}
Given $(q_0,v_0) \in \mathbb{R}^2$ with $q_0>0$ there exists a unique positive global solution $q\in W^{1,\infty}_{loc}(\mathbb R ^+)$
to \eqref{ode} with initial conditions:
\begin{equation} \label{CI}
q(0) = q_0 \quad \dot{q}(0) = v_0.
\end{equation}
\end{proposition}

Let us take some threshold value $q_{s}$.
We suppose that $q$ goes below this value after some time $t_1$
when $q(t)$ hits the threshold $q_{s}$ for the first time. Our aim is to find a suitable approximation
of $q$ after $t_1$ which enables to predict the time $t_2$ when $q$ goes above the threshold $q_s$
without solving \eqref{ode}. To this end, we introduce
\begin{equation}\label{vbar}
\bar{v}(t)=\dot{q}(t_{1})+\int_{t_{1}}^{t}g(s)ds
\end{equation}%
for $t\ge t_1$ and note that $\bar{v}(t_2)=\dot{q}(t_2)$. Indeed, integrating (\ref{ode}) from $t_1$ to $t>t_1$ shows that
\begin{equation}\label{odeii}
\dot{q}(t)=\dot{q}(t_1)+N(q_s)-N(q(t))+\int_{t_1}^{t}g(s)ds=N(q_s)-N(q(t))+\bar{v}(t).
\end{equation}%
Setting here $t=t_2$ and noting that $q(t_2)=q_s$ gives the desired result. Our approximation of the trajectory $q(t)$, denoted by $\bar{q} (t)$, stems from the assumption (verified afterwards in Proposition \ref{fact3}) that the velocity $\dot{q}(t_2)$ at the return point $t_2$ is small provided the threshold $q_s$ is small. Consequently, a good approximation for the time $t_2$ should be provided by the time $\bar{t}_{2}$ defined as the first time larger than $t_1$ when $\bar{v}(t)=0$. The construction of the approximated trajectory is hence the following: we first assume that $\bar{q}(t)$ is the same as $q(t)$ until the latter hits $q_s$ for the first time at $t=t_1$. Next, the trajectory $\bar{q}(t)$ is frozen until the time $t=\bar{t}_2$ and resumes then again as a solution to \eqref{ode} starting from $q_s$ with zero velocity:
\begin{equation}\label{qbar}
\bar{q}(t)=\left\{
\begin{array}{l}
q(t)\text{, for }0<t<t_{1} \\
q_{s}\text{, for }t_{1}\leq t<\bar{t}_{2}\\
\text{solution to (\ref{ode}) with }q(\bar{t}_{2})=q_{s}\text{, }\dot{q}(\bar{t}_{2})=0,\text{ for }t\geq \bar{t}_{2}
\end{array}
\right.
\end{equation}
(see Figure \ref{fig_constr_qbar}). 

\begin{figure}[!h]
     \hspace*{3cm}
 \psfrag{qbar}{ {$\overline{q}$}}
 \psfrag{q}{ $q$}
 \psfrag{qs}{ $q_s$}
 \psfrag{qtil}{ {$\tilde q$}}
 \psfrag{t1}{ $t_1$}
 \psfrag{t2}{ $t_2$}
 \psfrag{t2b}{ {$\overline{t}_2$}}
 \psfrag{t2t}{{$\widetilde{t}_2$}}
 
 \centerline{
  \includegraphics[scale=0.7,angle=-90]{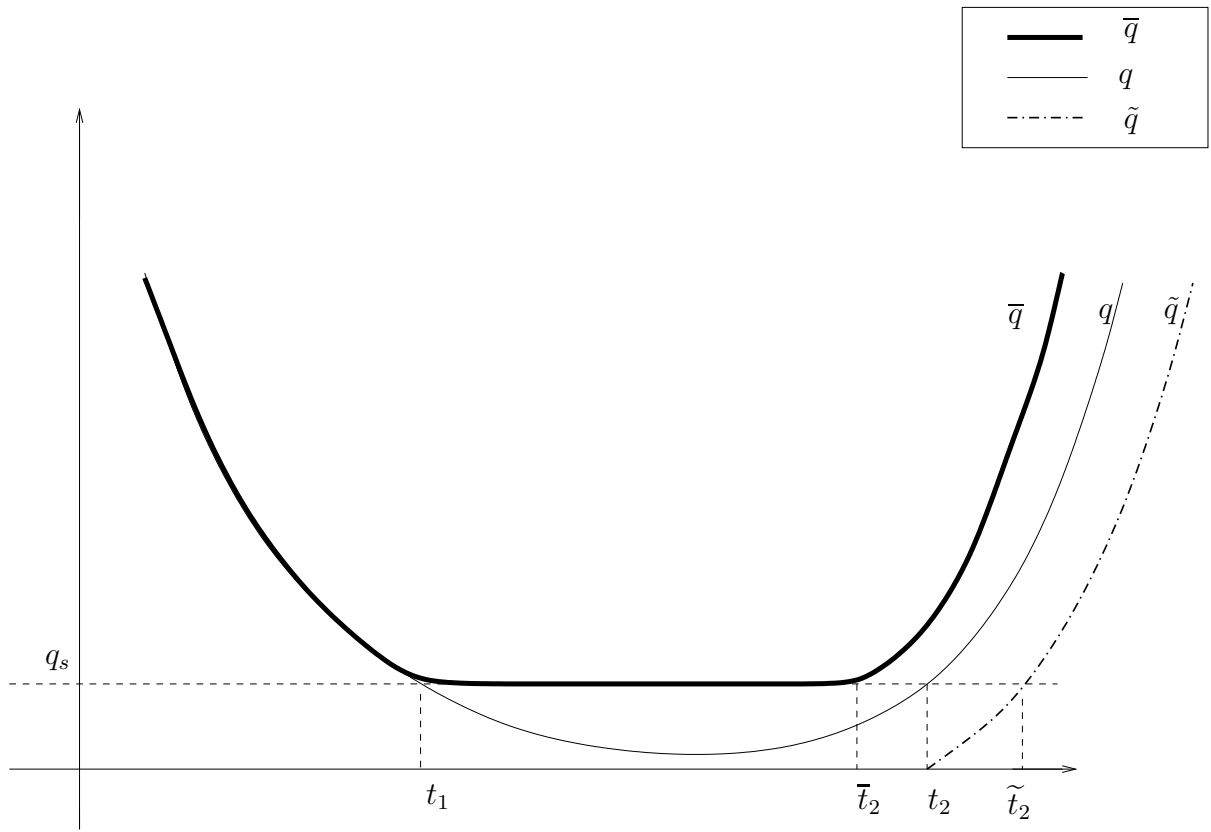}
  }
     \caption{Construction of the approximated trajectories $\bar{q}(t)$ and $\tilde{q}(t)$.}
     \label{fig_constr_qbar}
\end{figure}

In order to study the error committed by introducing the approximated trajectory $\bar{q}(t)$, we start by inferring the following estimate:
\begin{proposition}\label{fact1}
The time $\bar{t}_2$ provides a lower bound for $t_2$, { {\it \textit{i.e.}}} $\bar{t}_2\le t_2$.
\end{proposition}
\begin{proof}
{ {There are three cases. If $t_2$ does not exist, there is nothing to prove. If $t_2 = t_1$ this means $t_1$ is a local minimum for $q$ so
that $\dot{q}(t_1) = 0$ and $t_1 = \overline{t}_2 = t_2.$ Finally, if $t_2$ exists and satisfies $t_1 < t_2$,}}
since $q(t_1)=q(t_2)=q_s$ the function $q(t)$ has a local minimum $t_{\min}\in(t_1,t_2)$ where $\dot{q}(t_{\min})=0$. Evaluating (\ref{odeii}) at $t=t_{\min}$ shows then
$\bar{v}(t_{\min})=N(q(t_{\min}))-N(q_s)\le 0$. On the other hand, $\bar{v}(t_2)=\dot{q}(t_2)\ge 0$ since the function $q(t)$ is increasing at $t_2$ by the definition of $t_2$.
The intermediate value theorem tells now that $\exists\bar{t}\in[t_{\min},t_2]$ such that $\bar{v}(\bar{t})=0$. Thus the first time $t=\bar{t}_2$ when $\bar{v}(t)=0$
is certainly in $(t_1,\bar{t}]\subset (t_1,t_2]$.\\
\end{proof}

We have proved, in particular, that if the true trajectory $q(t)$ ever returns to the values above the threshold $q_s$, {\it \textit{i.e.}} $t_2<\infty$, the indicator $\bar{t}_2$ will show it, {\it \textit{i.e.}} $\bar{t}_2<\infty$. In Proposition \ref{fact2}, we compute an upper bound $\tilde{t}_2$ for $t_2$. 
In particular, under suitable assumptions on $g$ the time $t_2$ exists.

To see moreover that $\bar{t}_2$ may be a good approximation to $t_2$ we need, as mentioned before, to control $\dot{q}(t_2).$
The following proposition proves that $\dot{q}(t_2)$  is not too large with respect to $q_s.$
\begin{proposition}\label{fact3}
Let $\dot{q}^+$ (resp. $g^+$) be the positive part of $\dot{q}$ (resp. $g$): { {$\dot{q}^+=\max(\dot{q},0)$ (resp. $g^+=\max(g,0)$)}}. Then
\begin{equation}\label{est3}
\|\dot{q}^+\| \leq \frac {\|g^+\|}{n(q_s)},
\end{equation}
where $\|\cdot\|$ denotes the supremum on $[t_1,t_2]$. In particular, $\dot{q}(t_2)\ge 0$, so that $\dot{q}(t_2) \leq \frac {\|g^+\|}{n(q_s)} $.
\end{proposition}
\begin{proof}
Let $\dot{q}$ be non-negative on an interval $[t_{\min},t_{\max}]\subset[t_1,t_2]$. Without loss of generality, we can thus assume that $t_{\min}$ is
a local minimum of $q(t)$ and $t_{\max}$ is either a local maximum or $t_{\max}=t_2$.
Then, we introduce
$$
f(t) := \dot{q}(t) - \dfrac{\|g^+\|}{n(q(t))}, \quad \forall \, t \in [t_1,t_2].
$$
If $f(t_0)>0$  at some $t_0\in(t_{\min},t_{\max})$ then $f(t)$ rests positive on a small interval around $t_0$ 
so that $\ddot{q}< 0$ by invoking (\ref{ode}) and thus $\dot{q}(t)$ is decreasing on this interval. 
Since $q(t)$ is increasing on $[t_{\min},t_{\max}]$ and $n(q)$ is decreasing, we obtain that
$f(t)$ is decreasing locally around $t_0$ so that $f(t)\ge f(t_0)>0$ for all $t\in[t_0-\delta,t_0]$ with some $\delta > 0$.  
Repeating this argument for times down to $t_{\min}$ we see that $f(t)$ should be positive on $[t_{\min},t_0]$ but this is
impossible since $\dot{q}(t_{\min}) = 0$ and $f(t_{\min}) < 0$. This proves that $f(t) \le 0$ on $[t_{\min},t_{\max}]$.
\end{proof}

The results above allow us to fully characterize the error of
approximating $q(t)$ by $\bar{q}(t)$ up to time $\bar{t}_2$ and to
bound it by certain quantities depending on $\bar{q}$ alone for
time after $\bar{t}_2$, at least in the special case when $g(t)$
is given by
\begin{equation}\label{gspec}
 g(t)=
    \begin{cases}
    g_-(t)< 0,\ &\textrm{for}\ t\le t_0,\\
    g_+>0,\ &\textrm{for}\ t> t_0,\\
    \end{cases}
\end{equation}
with some positive constants $t_0$, $g_+$ and a negative function $g_-(t)$. Indeed, we prove in the following proposition 
that the error $\tau=t_2-\bar{t}_2$ is bounded by $1/n(q_s)$ and we remind that $n(q_s)\to\infty$ as $q_s\to 0$.
\begin{proposition}\label{cor}
If $g(t)$ is given by \eqref{gspec}, then
\begin{align}
&0\le t_2-\bar{t}_2 \le \frac{1}{n(q_s)},\notag\\
&\sup_{t\in[0,\bar{t}_2]}|q(t) - \bar{q}(t)| \le q_s, \label{eqcor}\\
&|\dot{q}(\bar{t}_2) - \dot{\bar{q}}(\bar{t}_2)| = \dot{q}(\bar{t}_2) \le \frac{g_+}{n(q_s)}. \notag
\end{align}
Moreover, both $q(t)$ and $\bar{q}(t)$ are non-decreasing for
$t\ge\bar{t}_2$ and
\begin{equation}\label{eqcor2}
 |q(t)-\bar{q}(t)|\le
    \begin{cases}
    \bar{q}(t),\ &\textrm{if}\ \bar{t}_2\le t < t_2,\\
    \bar{q}(t)-\bar{q}(t-\tau),\ &\textrm{if}\ t \ge t_2,\\
    \end{cases}
\end{equation}
with $\tau=t_2-\bar{t}_2$.
\end{proposition}
\begin{proof}
We first note that $t_2\ge\bar{t}_2\ge t_0$. The first inequality here is already proved in Proposition \ref{fact1}. The second inequality follows from $\bar{v}(\bar{t}_2)=0$, which
can be rewritten as
$$
-\dot{q}(t_1)=\int_{t_1}^{\bar{t}_2}g(s)ds.
$$
Since $\dot{q}(t_1)\le 0$, the last equality cannot hold if $g$ is negative everywhere on $[t_1,t_2]$. Since $\bar{t}_2\ge t_0$, we have $g(t)=g_+>0$ for $t\ge\bar{t}_2$. Evaluating \eqref{odeii} at $t=t_2$ gives
\begin{equation}\label{dqt2}
\dot{q}(t_2)=\int_{\bar{t}_2}^{t_2}g(s)ds = g_+(t_2-\bar{t}_2).
\end{equation}
Recalling Proposition \ref{fact3}, we see that
$$
t_2-\bar{t}_2 = \frac{\dot{q}(t_2)}{g_+} \le \frac{1}{n(q_s)},
$$
which is the first inequality in \eqref{eqcor}. The second inequality in \eqref{eqcor} is obvious since $0<q(\bar{t}_2)\le q_s$ and $\bar{q}(\bar{t}_2)=q_s$.
The third inequality follows from Proposition \ref{fact3} since $\dot{q}(\bar{t}_2)\ge 0$. Indeed, \eqref{odeii} evaluated at $t=\bar{t}_2$ gives $\dot{q}(\bar{t}_2)=N(q_s)-N(q(\bar{t}_2))$.

We now turn to the study of $q(t)$ and $\bar{q}(t)$ for time
$t\ge\bar{t}_2$, in order to prove that $q(t)$ and $\bar{q}(t)$ are non-decreasing on this interval. We first observe that $q(t)$ cannot have local
maximums at $t>t_0$. Indeed, at any local extremum $t_e>t_0$,
equation \eqref{ode} would imply $\ddot{q}(t_e)=g_+>0$. Since
$q(t)$ is non-decreasing at $\bar{t}_2>t_0$, it should be
non-decreasing also everywhere on $[\bar{t}_2,\infty)$. The same
reasoning applies to $\bar{q}(t)$. The estimate \eqref{eqcor2} is
now evident in the case $\bar{t}_2\le t \le t_2$ since
$\bar{q}(t)\ge q_s$ and $q(t)\le q_s$ for such $t$. 

In order to prove \eqref{eqcor2} in the case $t\ge t_2$ we will show that $q(t)$ is squeezed between $\bar{q}_\tau(t)$ and $\bar{q}(t)$ for $t\ge t_2$ where 
$\bar{q}_\tau(t)=\bar{q}(t-\tau)$. Indeed both $\bar{q}_\tau(t)$ and $q(t)$ satisfy the same equation \eqref{ode} with the same initial condition at $t=t_2$ 
($q(t_2)=\bar{q}_\tau(t_2)=q_s$) but possibly different initial velocities $\dot{q}(t_2)\ge 0$, $\dot{\bar{q}}_\tau(t_2)=0$. Writing \eqref{ode} for $\bar{q}_\tau(t)$ and $q(t)$,
taking the difference thereof and integrating from $t_2$ to $t$ yields
$$
\dot{q}(t)-\dot{\bar{q}}_\tau(t) = \dot{q}(t_2) - N(q(t)) + N(\bar{q}_\tau(t)). 
$$
This shows that if the trajectories of $q(t)$ and $\bar{q}_\tau(t)$ intersect at some time $t'$ { {then}} $\dot{q}(t')\ge\dot{\bar{q}}_\tau(t')$ so that $q(t)\ge\bar{q}_\tau(t)$ at least for some time 
after $t'$. Since $q(t)\ge\bar{q}_\tau(t)$ holds also for some time after $t_2$ it should hold for all $t\in[t_2,\infty)$. 

It remains to prove that $q(t)\le\bar{q}(t)$ for all $t\in[t_2,\infty)$. To this end, we integrate \eqref{ode} for $q(t)$ from $t_2$ to $t$ and that for $\bar{q}(t)$ from $\bar{t}_2$ to $t$.
This yields with the aid of \eqref{dqt2}
\begin{align}
&\dot{q}(t) = \dot{q}(t_2) + N(q_s) - N(q(t)) + \int_{t_2}^t g(s)ds
= N(q_s) - N(q(t)) + \int_{\bar{t}_2}^t g(s)ds, \label{dq1}\\
&\dot{\bar{q}}(t) =  N(q_s) - N(\bar{q}(t)) + \int_{\bar{t}_2}^t g(s)ds.  \label{dq2}
\end{align}
We see now that if $q(t')=\bar{q}(t')$ at some time $t'\ge\bar{t}_2$ { {then}} also $\dot{q}(t')=\dot{\bar{q}}(t')$. It means by the uniqueness of solutions to \eqref{ode} that the trajectories 
of $q(t)$ and $\bar{q}(t)$ either coincide or do not intersect on $[\bar{t}_2,\infty)$. Since $q(\bar{t}_2)\le\bar{q}(\bar{t}_2)$ this implies $q(t)\le\bar{q}(t)$ on $[\bar{t}_2,\infty)$.\\
\end{proof}

It is possible to relax the hypotheses of the last Proposition on the particular form of the function $g$ in several ways, 
if we still remain in the case when $q$ goes below the threshold value $q_s$ a finite number of times (in particular, $g$ should be allowed to change sign only a finite number of times). 
Although we do not have an analogue of these results for a general $g(t)$, we can always provide an easily computable upper bound for $t_2$ alongside the lower bound $\bar{t}_2$.

\begin{proposition}\label{fact2}
Let $\tilde{t}_{2}$ be the first time after $\bar{t}_2$ such that
\begin{equation}\label{tilde2}
\int_{\bar{t}_{2}}^{\tilde{t}_{2}}(\tilde{t}_{2}-s)g(s)ds=q_{s}.
\end{equation}
Then $\tilde{t}_{2}\ge t_2$.
\end{proposition}
\begin{proof}
Consider $\tilde{q}(t)$ defined for $t\geq \bar{t}_{2}$ as the solution
to
\begin{equation*}
\dot{\tilde{q}}(t)=\int_{\bar{t}_{2}}^{t}g(s)ds,~\ \tilde{q}(\bar{t}_{2})=0
\end{equation*}%
We { {recall}} now the property \eqref{dq1}  of the original solution and note $N(q(t))<N(q_{s})$ for $t\in (\bar{t}_{2},t_{2})$ so that 
\begin{equation*}
\dot{q}(t) \ge \int_{\bar{t}_{2}}^{t}g(s)ds \geq \dot{\tilde{q}}(t).
\end{equation*}%
Since $q(\bar{t}_{2})>\tilde{q}(\bar{t}_{2})$, the trajectory of $q(t)$ lies above that of $\tilde{q}(t)$ for $t\in (\bar{t}_{2},t_{2})$ so that $q(t)$ hits the threshold $q_{s}$ before $%
\tilde{q}(t)$. In other words, $t_{2}<\tilde{t}_{2}$, where $\tilde{t}_{2}$ is the first time after $\bar{t}_2$ such that $\tilde{q}(t)=q_{s}$. This definition 
of $\tilde{t}_{2}$ is equivalent to that in \eqref{tilde2}. 
\end{proof}

\section{Discretization in time of the model ODE}

\subsection{Three schemes for ODE (\ref{ode})}

We first describe a straightforward Euler discretization in time of (\ref{ode}) rewritten as the system (\ref{syst}). Introducing the time step $\Delta t$ and the discrete times $t_k=k\Delta t$, 
$k=1,2,\ldots$ we thus consider the following

\begin{algorithm}[h!] 
\caption{(straightforward Euler)} 
\label{alg1}
 \begin{deflist}[Step 2.]
   \item[\rm Step 0.] Initialize ($v^0$, $q^0$).
   \item[\rm Step 1.] For $k=1,2,\ldots$ update ($v^k$, $q^k$) as follows
     \begin{equation}\label{step1}
     \begin{cases}
     \frac{v^k-v^{k-1}}{\Delta t} = -n(q^{k-1})v^k+g(t_k) \\
     \frac{q^k-q^{k-1}}{\Delta t}= v^k.
     \end{cases}
     \end{equation}
 \end{deflist}
\end{algorithm}

Algorithm \ref{alg1} is inspired
by the real fluid-particle simulations in which one first finds the new
velocity at each time step and then moves the particle with this velocity.
An immediately evident drawback of the scheme in Algorithm \ref{alg1} is that it does not necessarily provide a positive approximation $q^k$. We remind that negative
values of $q$ are unphysical. Moreover, even if approximations $q^k$ remain positive but become tiny at some time steps, one would require very small stepsize to obtain an accurate solution.

\begin{remark}
One may argue that this can be cured by resorting to a fully implicit scheme that couples the evaluation of the
velocity with the displacement of the particle:
     \begin{equation}\label{impl}
     \begin{cases}
     \frac{v^k-v^{k-1}}{\Delta t} = -n(q^{k})v^k+g(t_k) \\
     \frac{q^k-q^{k-1}}{\Delta t}= v^k.
     \end{cases}
     \end{equation}
Indeed, this scheme provides a positive solution for $q$, for example, in the important case $n(q)=\varepsilon/q$. 
To see this we eliminate $v^{k }$ from \eqref{impl}, which gives an equation for $q^k$ only
\[
\frac{q^{k}}{{\Delta t} }-\frac{\varepsilon q^{k-1}}{q^{k}}=\frac{q^{k-1}}{{\Delta t} }%
-\varepsilon +v^{k-1}+{\Delta t} g(t_{k}) 
\]%
The function of $q^{k}$ in the left-hand side is increasing form $-\infty $
to $\infty $ when $q^{k}$ goes from 0 to $\infty $. This means that there is
the unique positive solution $q^{k}$ for any given $v^{k-1}$ and $q^{k-1}$. However, scheme \eqref{impl} would be too expensive in a real fluid/particle simulation
where $n(q^k)$ should be recomputed after any change in $q^k$ through a numerical approximation of the Stokes equations.
\end{remark}

\begin{remark}
Difficulties related to the discretization of these differential equations have already been pointed out in \cite{Maury2007} (see Remark 3.1): 
it is shown that very small values of $q$ can be reached (below the smallest real number which can be stored by standard numerical softwares and also below intermolecular distances). 
Different approaches have been proposed in previous related works to deal with similar problems: in \cite[Section 5.2]{Maury2007}, the author suggests
the introduction of a cut-off function for the microscopic distance $\gamma=\varepsilon \ln (q)$, defined in terms of the roughness of the surface (see also \cite[Section 2.5]{Lefebvre09}); 
in \cite[Section 2.4.3]{Lefebvre09}, for the case of fluid/particle simulation, a constraint for the distance $q$ is defined in terms of the mesh size of the fluid domain. 
\end{remark}

In order to overcome these difficulties, in this work we propose to introduce a small threshold value $q_s>0$, to replace the exact solution $q(t)$ by the approximated one $\bar{q}(t)$, as summarized 
in \eqref{qbar} on the continuous level. Discretization in time of this idea introducing also an approximation of $\bar{v}(t)$ defined in \eqref{vbar} is detailed in Algorithm \ref{alg2}. 
Note that the exact threshold is replaced in Step 2 of this algorithm by the last available value of $q^k$ before passing below $q_s$, { {\it \textit{i.e.}}} $q^{k_1-1}$, so that the passage time $t_1$
is legitimately approximated by $t_{k_1-1}$.

\begin{algorithm}[h!] 
\caption{(with an {\it a priori} fixed threshold)} 
\label{alg2}
 \begin{deflist}[Step 2.]
   \item[\rm Step 0.] Initialize ($v^0$, $q^0$) and choose $q_s>0$.
   \item[\rm Step 1.] For $k=1,2,\ldots$ update ($v^k$, $q^k$) by \eqref{step1}
   until $q^k\leq q_s$ at some $k=k_1$. We then redefine $v^{k_1}=0$, $q^{k_1}=q^{k_1-1}$, initialize $\bar{v}$ by $\bar{v}^{k_1}=v^{k_1-1}+ g(t_{k_1})\Delta t$ and switch to Step 2.
 \item [\rm Step 2.] For $k=k_1+1,k_1+2,\ldots$ update $\bar{v}^k$ as follows
 \begin{equation}\label{step2}
 \bar{v}^k=\bar{v}^{k-1} + g(t_k)\Delta t
 \end{equation}
 keeping $q^{k}=q^{k_1-1}$, $v^{k}=0$ until $\bar{v}^k \geq 0$ at some $k=k_2$. We then abandon the calculation of $\bar{v}$ and switch to Step 3.
 \item [\rm Step 3.] For $k=k_2+1,k_2+2,\ldots$ update ($v^k$, $q^k$) as in Step 1. If $q^k$ goes again under $q_s$ at some time step, then reintroduce $\bar{v}$ and 
 keep switching between Steps 1 and 2 back and forth.
\end{deflist}
\end{algorithm}

As suggested by Proposition \ref{cor} a good choice for $q_s$ would be such that $1/n(q_s)=C\Delta t$ with some constant $C$ so that the criterion for switching from Step 1 to Step 2 can be rewritten as 
$n(q^k)\ge 1/(C\Delta t)$. Indeed, extrapolating the result of Proposition \ref{cor} to a general $g(t)$, we see then that the error of approximation of the exact return time $t_2$ 
by $\bar{t}_2\approx k_2\Delta t$ is of order $\Delta t$. Moreover, the error $|\bar{q}(t)-q(t)|$ is dominated by $|\bar{q}(t)-\bar{q}(t-\tau)|$ with $\tau$ of order $\Delta t$. 
Thus the error caused by the introduction of the threshold should be of of the same order as that of Euler scheme itself. However, an optimal choice of the constant $C$ above would 
require some \textit{a posteriori} error indicators to control the error of the Euler discretization. 

\begin{algorithm}[h!] 
\caption{(with adapted variable stepsize and automatically chosen threshold )} 
\label{alg3}
 \begin{deflist}[Step 2.]
   \item[\rm Step 0.] Initialize ($v^0$, $q^0$) for $t_0=0$, choose the first time step $\Delta t_1>0$, the minimal tolerated time step $\Delta t_{\min}>0$ and the error tolerance 
   per time step $tol>0$.
   \item[\rm Step 1.] For $k=1,2,\ldots$ set $t_k=t_{k-1}+ \Delta t_k$, update ($v^k$, $q^k$) by
     \begin{equation}\label{step1k}
     \begin{cases}
     \frac{v^k-v^{k-1}}{\Delta t_k} = -n(q^{k-1})v^k+g(t_k) \\
     \frac{q^k-q^{k-1}}{\Delta t_k}= v^k.
     \end{cases}
     \end{equation}
     and calculate the approximation for the error committed on this time step by $e_k=\max(|\ddot q(t_k)|,|\ddot v(t_k)|)\Delta t_k^2/2$ where $\ddot q(t_k)$, $\ddot v(t_k)$
     are computed using \eqref{ddotq}--\eqref{ddotv}. If $e_k \le tol$ and $q_k>0$ then we accept the just computed values ($v^k$, $q^k$) and proceed to the next time step 
     increasing the time step to $\Delta t_{k+1}=\sqrt{\frac{tol}{e_k}}\Delta t_k$. Otherwise, if $e_k > tol$ or $q_k \le 0$, we reject the approximation ($v^k$, $q^k$) and try 
     to recompute it by \eqref{step1k} with a smaller time step $\Delta t_k:=\min(\sqrt{\frac{tol}{e_k}}\Delta t_k,\Delta t_k/2)$. If the new approximation ($v^k$, $q^k$) is still 
     not sufficiently accurate, we try to diminish the time step again and again until an acceptable approximation ($v^k$, $q^k$) is computed and proceed only then to the next time step taking
     $\Delta t_{k+1}$ equal to the smallest value of $\Delta t_k$ used on step $k$. However, if in the process of reducing $\Delta t_k$ we come to a time step smaller than $\Delta t_{\min}$
     at some $k=k_1$,  we abandon the calculation of ($v^{k_1}$, $q^{k_1}$) and switch to Step 2 taking the initial values $\bar{v}^{k_1-1}=v^{k_1-1}$, $q_s=q^{k_1-1}$ and setting $\Delta t$ 
     to the current value of $\Delta t_k$.
  \item [\rm Step 2.] For $k=k_1,k_1+1,\ldots$ update $\bar{v}^k$ by \eqref{step2} keeping $q^{k}=q_s$, $v^{k}=0$ until $\bar{v}^k \geq 0$ at some $k=k_2$. 
  We then abandon the calculation of $\bar{v}$ and switch to Step 3.
 \item [\rm Step 3.] For $k=k_2+1,k_2+2,\ldots$ update ($v^k$, $q^k$) as in Step 1. If $q^k$ goes again under $q_s$ at some time step, then reintroduce $\bar{v}$ and 
 keep switching between Steps 1 and 2 back and forth.
 \end{deflist}
\end{algorithm} 

Generally speaking, a discretization of an ODE using the constant time step is often not optimal concerning the CPU time needed to achieve a desired accuracy. This is especially true 
for our ODE \eqref{ode} since the coefficient $n(q(t))$ can change enormously during a simulation. One could try therefore to modify our algorithms by introducing a non constant stepsize
chosen at each step using an error indicator. A general recipe for an optimal adaptation of the stepsize, according to \cite{Bu03} is to choose the stepsizes so that the discretization 
error per time step remains approximately constant during the whole simulation.  As is well known, the error per time step in a Euler scheme like that of Algorithm 1, is given to the leading order 
by $\max(|\ddot q(t_k)|,|\ddot v(t_k)|)\Delta t_k^2/2$ where $\Delta t_k=t_k-t_{k-1}$ is the stepsize on step $k$ which can be varying. Assuming that ($q^{k-1}$, $v^{k-1}$) and ($q^k$, $v^k$)
are Euler approximations at times $t_{k-1}$ and $t_k$ respectively, the second derivatives of $q$ and $v$ can be computed approximately by deriving the equations in \eqref{syst} with respect to time
and replacing the missing derivatives by finite differences. This gives
\begin{align}
\ddot q(t_k) &\approx -n(q^k)v^k+g(t_k) \label{ddotq}\\
\ddot v(t_k) &\approx n^2(q^k)v^k-n(q^k)g(t_k)-\frac{n(q^k)-n(q^{k-1})}{\Delta t_k}v^k+\frac{g(t_k)-g(t_{k-1})}{\Delta t_k} \label{ddotv}
\end{align}
Thus, denoting $E_k=\max(|\ddot q(t_k)|,|\ddot v(t_k)|)/2$ where the second derivatives are replaced by the formulas above, the strategy to choose the time step would be to require
$E_k\Delta t_k^2=tol=const$ where {\it tol} is prescribed tolerance. As this recipe can lead to extremely small $\Delta t_k$ we combine it with the threshold approximation by 
switching to $\bar{q}(t)$ only when $\Delta t_k$ becomes smaller than some minimal admissible time step size. This idea is implemented in Algorithm \ref{alg3}.

\subsection{Some numerical tests}

As an illustration, we consider the ODE \eqref{ode} with $n(q)=\varepsilon/q^{3/2}$ corresponding to the lubrication force in 2D. 
We report several numerical results setting
$$
 g(t)=
    \begin{cases}
    -2,\ &\textrm{for}\ t\le 2,\\
    2,\ &\textrm{for}\ t> 2,\\
    \end{cases}
$$
and the initial conditions $q(0)=1$, $\dot{q}(0)=0$. In all the cases, we have at our disposal a very accurate reference solution, which we call { {``exact''}} in what follows.

We first compare Algorithm 1 (without threshold) vs. Algorithm 2 with the threshold such that $n(q_s)=1/(20\Delta t)$. The choice of the coefficient $20$ in the last formula is somewhat arbitrary, but 
it works fine in our test cases. The first series of numerical experiments is performed taking $\varepsilon=0.1$.  The results are reported in Figure \ref{eps01}. 
Note that, although the solution obtained with Algorithm 1 is positive at $\Delta t = 0.01$ and smaller, the introduction of the threshold in Algorithm 2 seems to enhance the quality 
of the solution at large time. 

\begin{figure}[h!]
 \centerline{
     \includegraphics[scale=0.35]{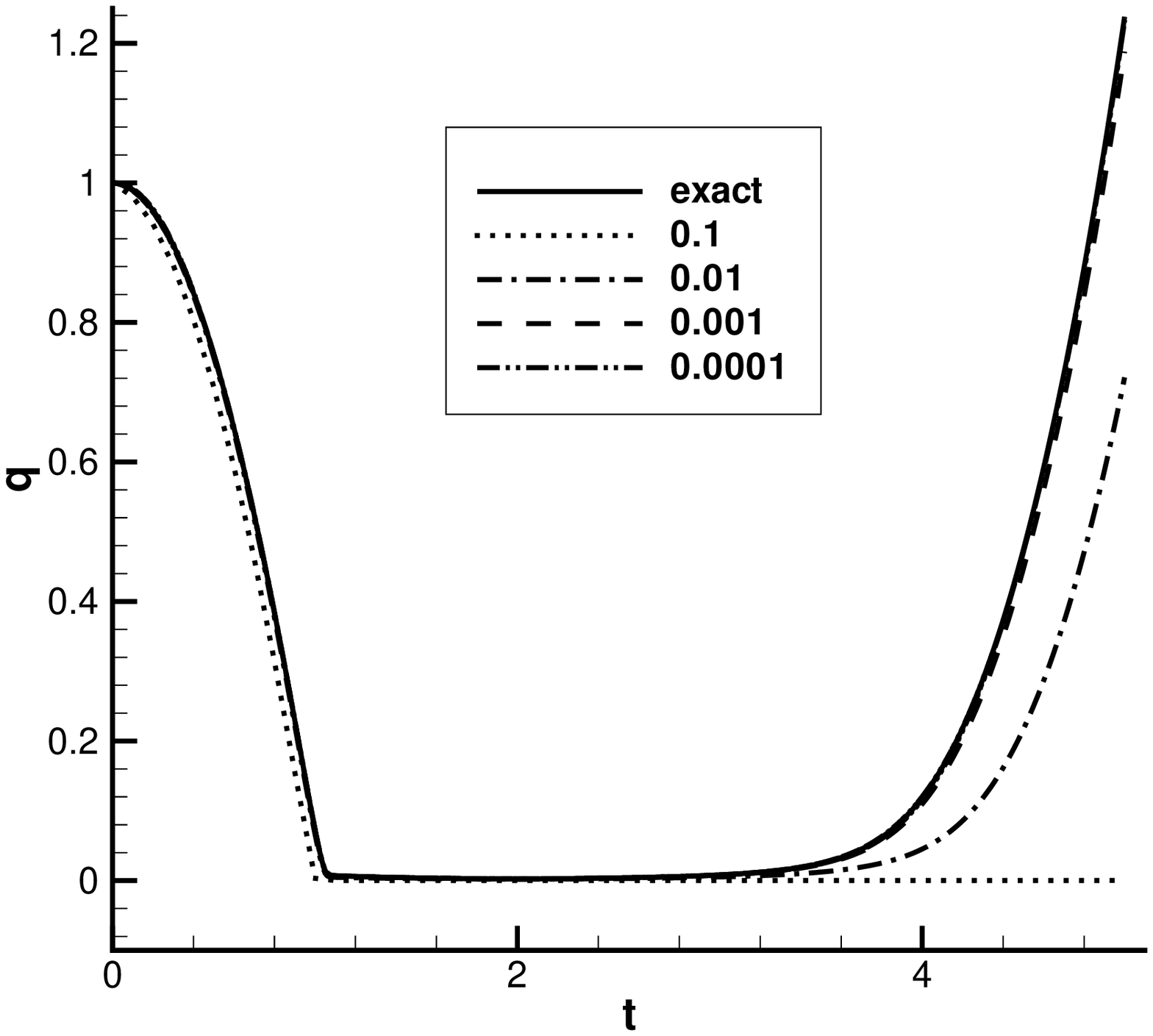}
     \qquad
     \includegraphics[scale=0.35]{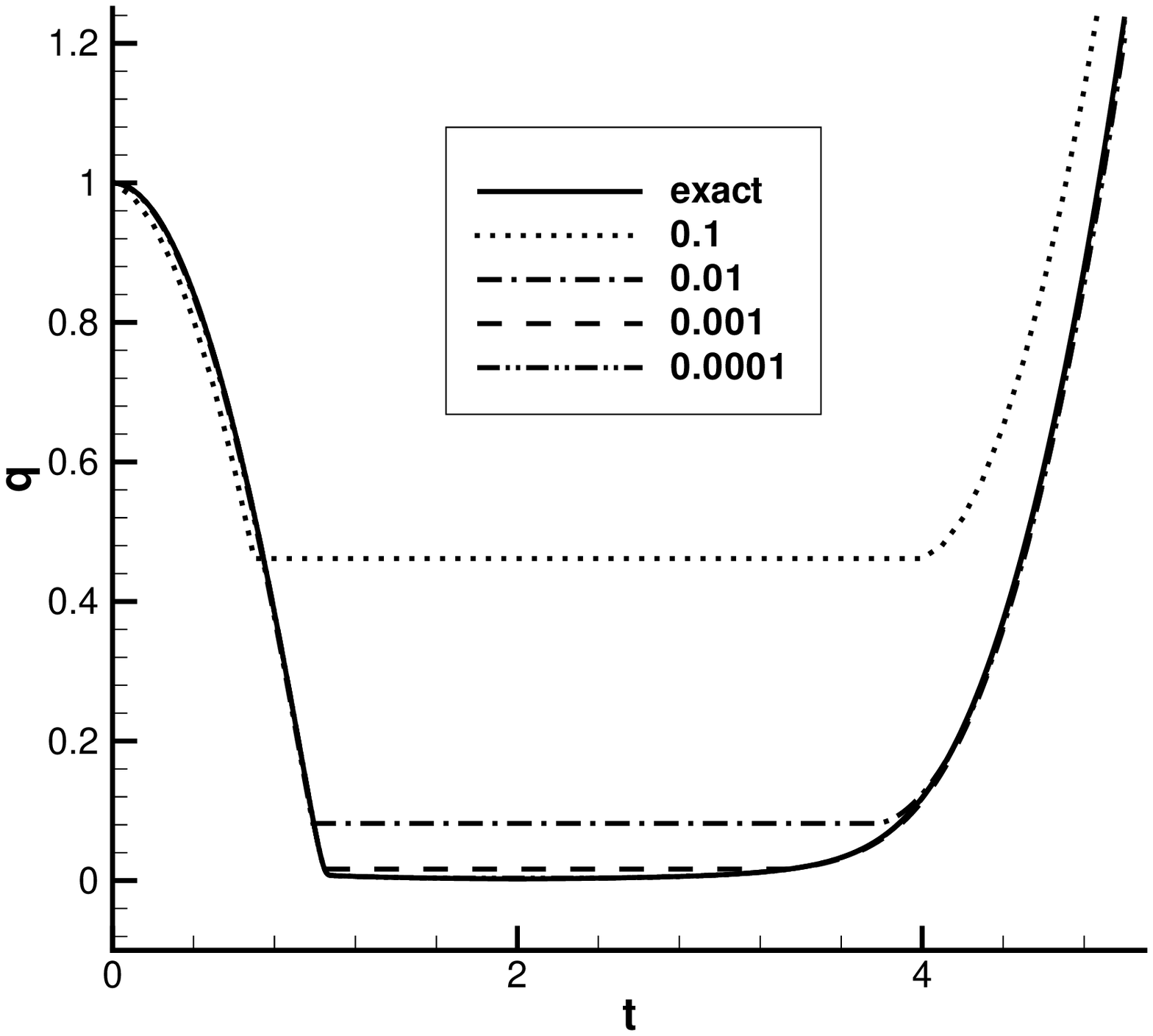}
 }
 \caption{Solution to the ODE \eqref{ode} with $n(q)=\varepsilon/q^{3/2}$, $\varepsilon=0.1$. Left: the exact solution and solutions obtained by Algorithm 1 with $\Delta t$ ranging from 
 $0.1$ to $0.0001$. Right: solutions obtained by Algorithm 2 with $\Delta t$ in the same range. The results obtained with $\Delta t=0.0001$ are visually indistinguishable from
 the exact solution.}
 \label{eps01}
\end{figure}

\begin{figure}[h!]
 \centerline{
    \includegraphics[scale=0.35]{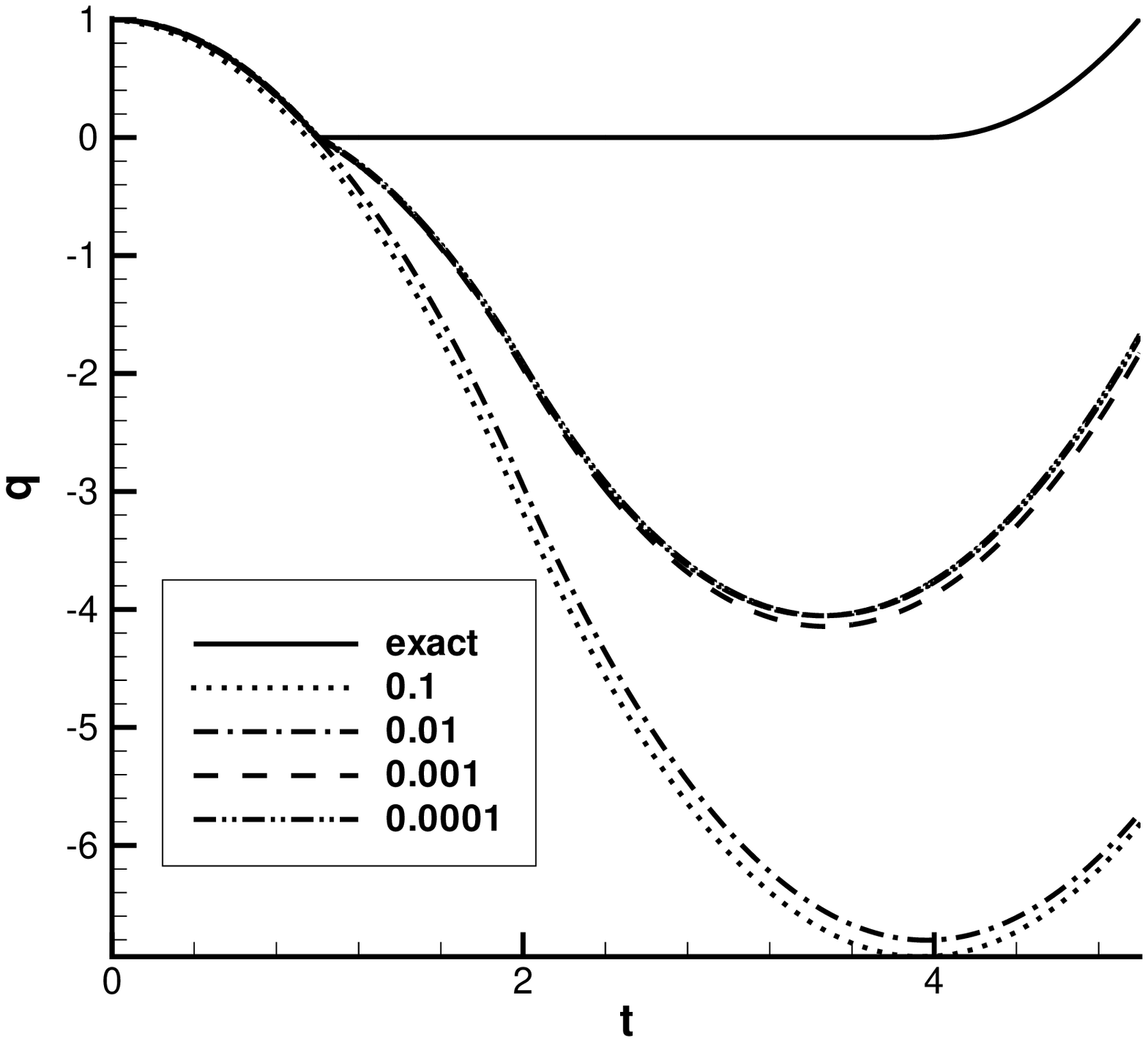}
     \qquad
    \includegraphics[scale=0.35]{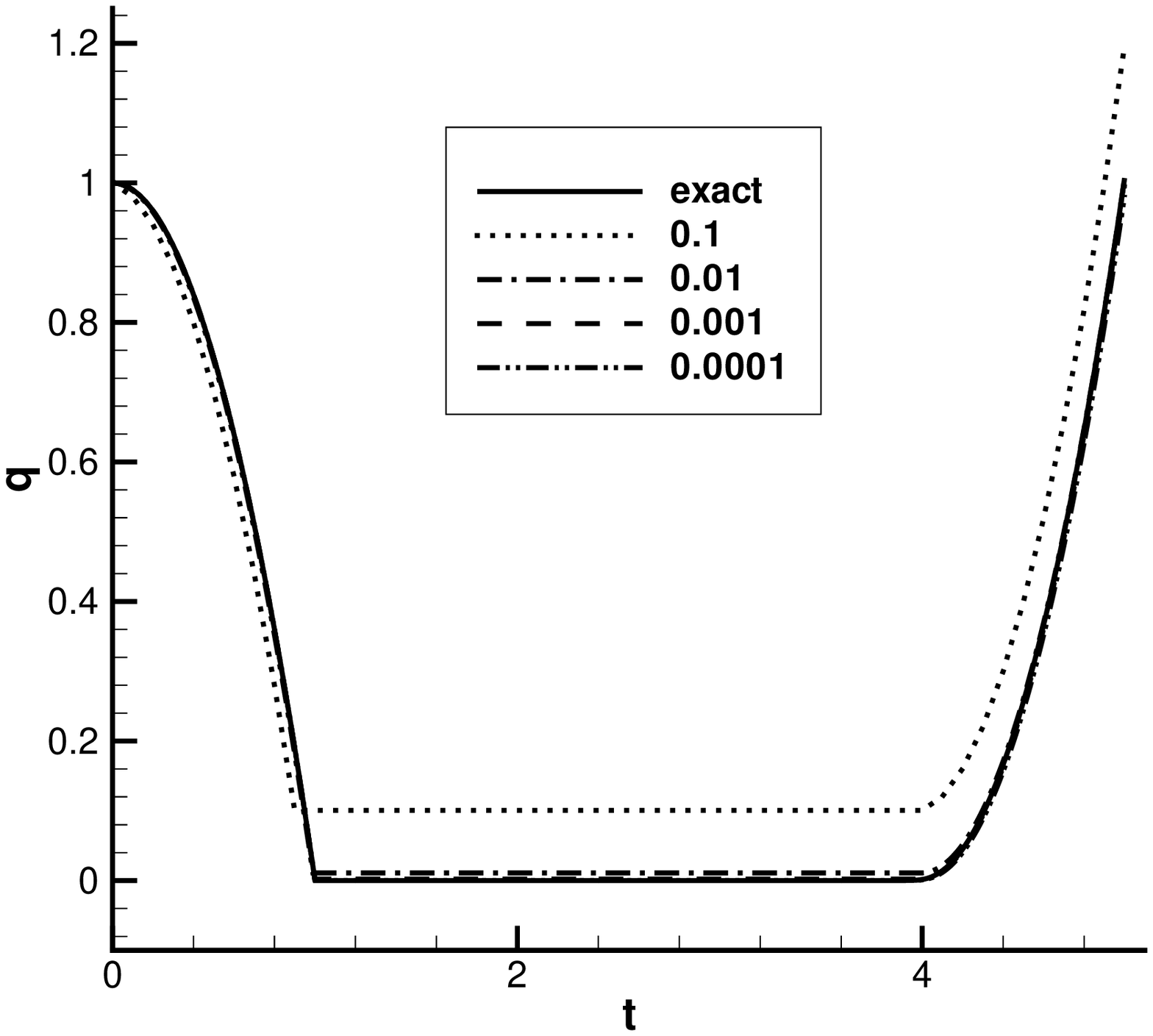}
 }
 \caption{Solution to the ODE \eqref{ode} with $n(q)=\varepsilon/q^{3/2}$, $\varepsilon=10^{-3}$. Left: the exact solution and solutions obtained by Algorithm 1 with $\Delta t$ ranging from 
 $0.1$ to $0.0001$. Right: solutions obtained by Algorithm 2 with $\Delta t$ in the same range. 
 The results obtained with $\Delta t\le 0.01$ on the right are visually indistinguishable from the exact solution.
 }
 \label{eps0001}
\end{figure}

\begin{figure}[h!]
 \centerline{
     \includegraphics[scale=0.35]{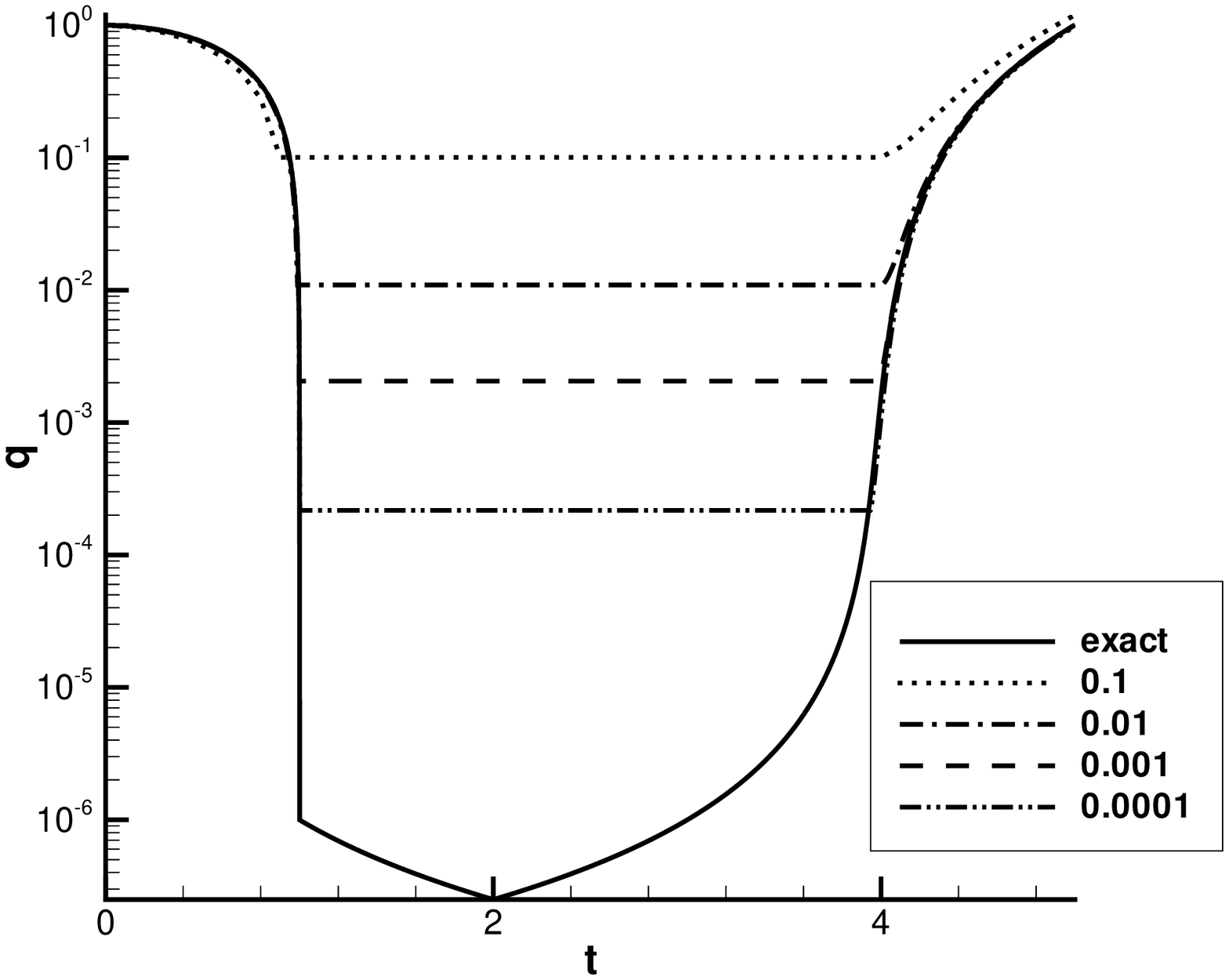}
      \qquad
     \includegraphics[scale=0.35]{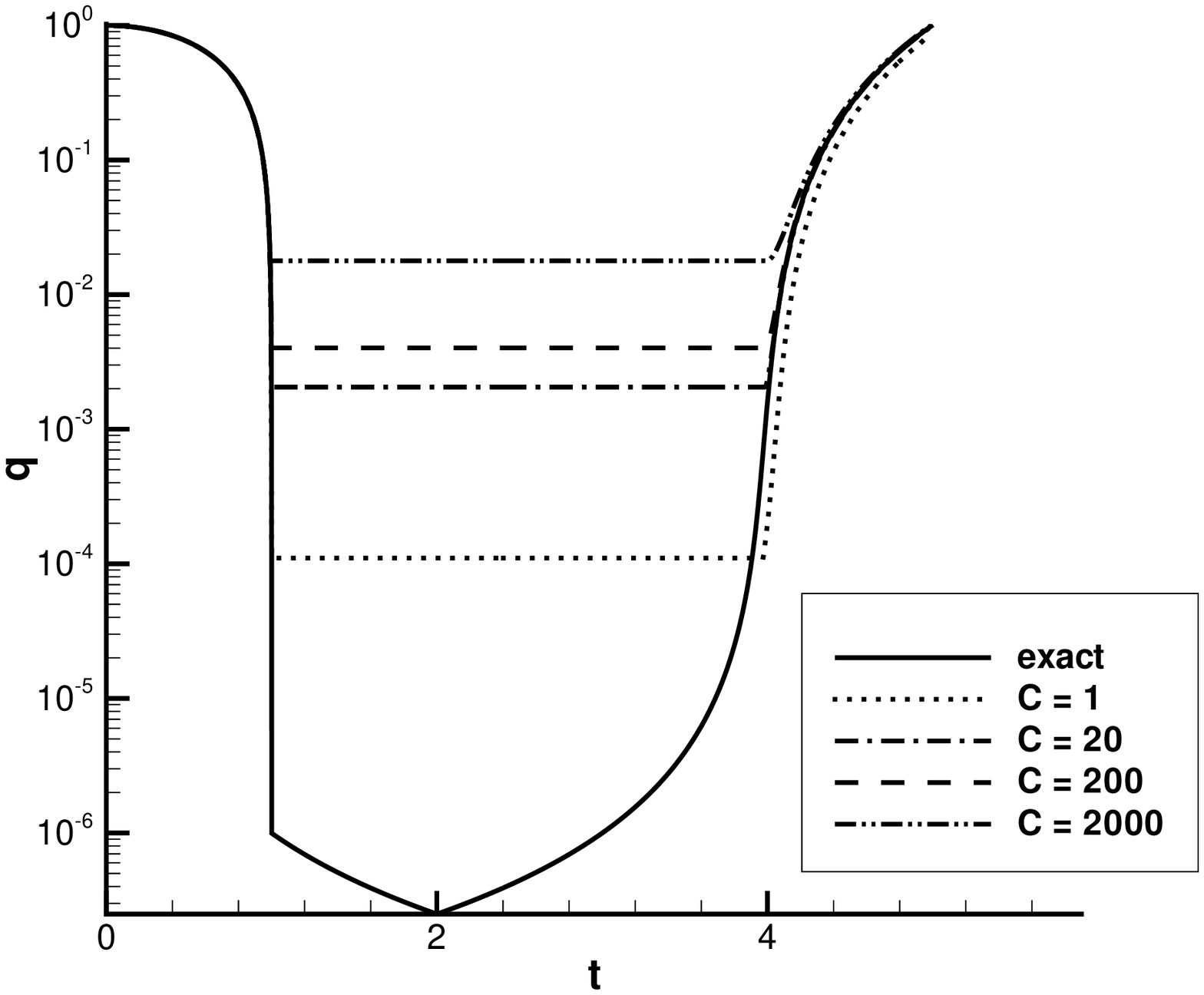}
 }
 \caption{Left: same results as on Figure \ref{eps0001}, \textit{i.e.} varying $\Delta t$ and choosing $q_s$ so that $n(q_s)=1/(20\Delta t)$.
 Right: same test with $\Delta t=0.001$ fixed and different values of $q_s$ calculated so that $n(q_s)=1/(C\Delta t)$ and $C$ varying in the range $[1,2000]$.
 }
 \label{eps0001log}
\end{figure}

\begin{figure}[h!]
 \centerline{
     \includegraphics[scale=0.35]{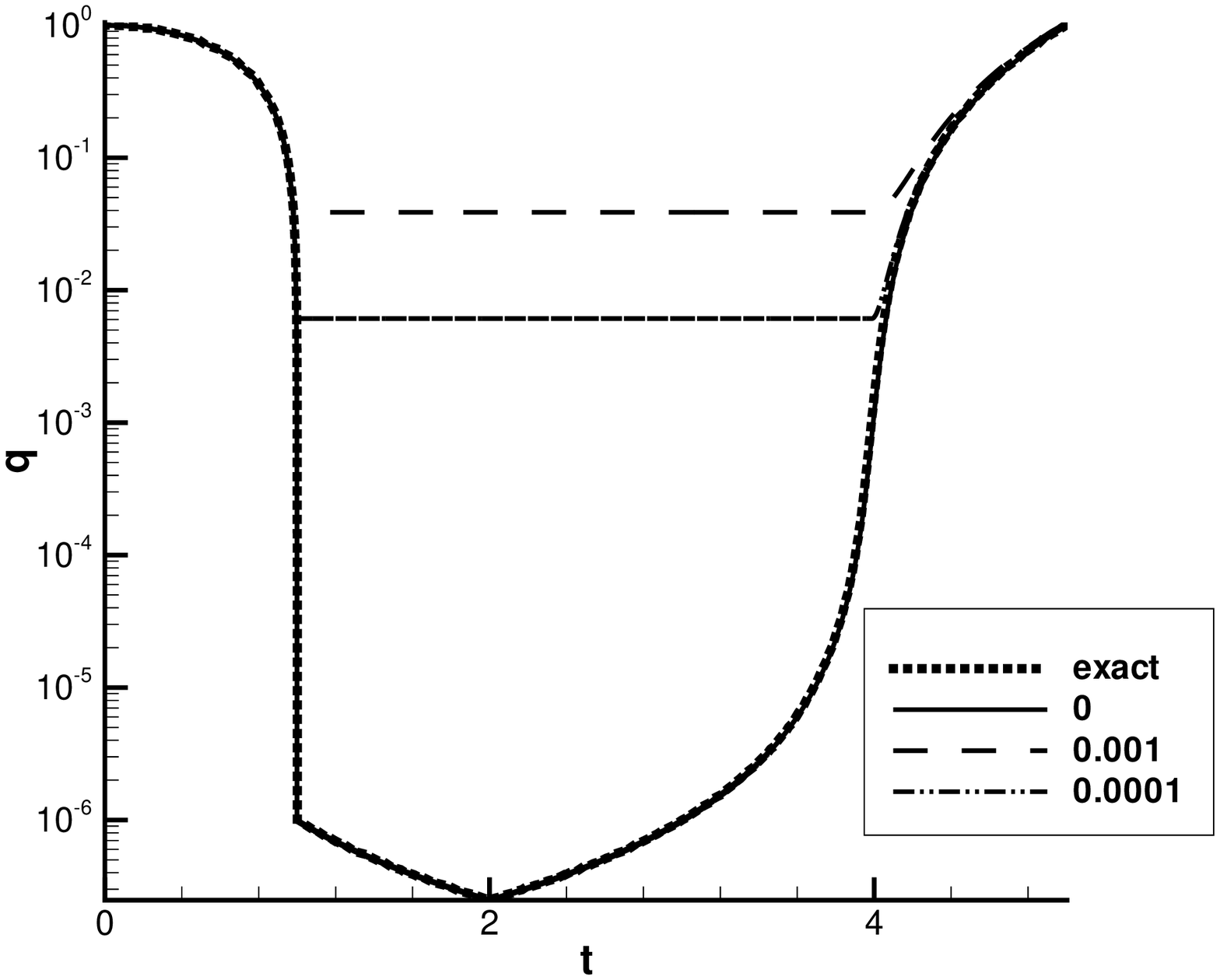}
      \qquad
     \includegraphics[scale=0.35]{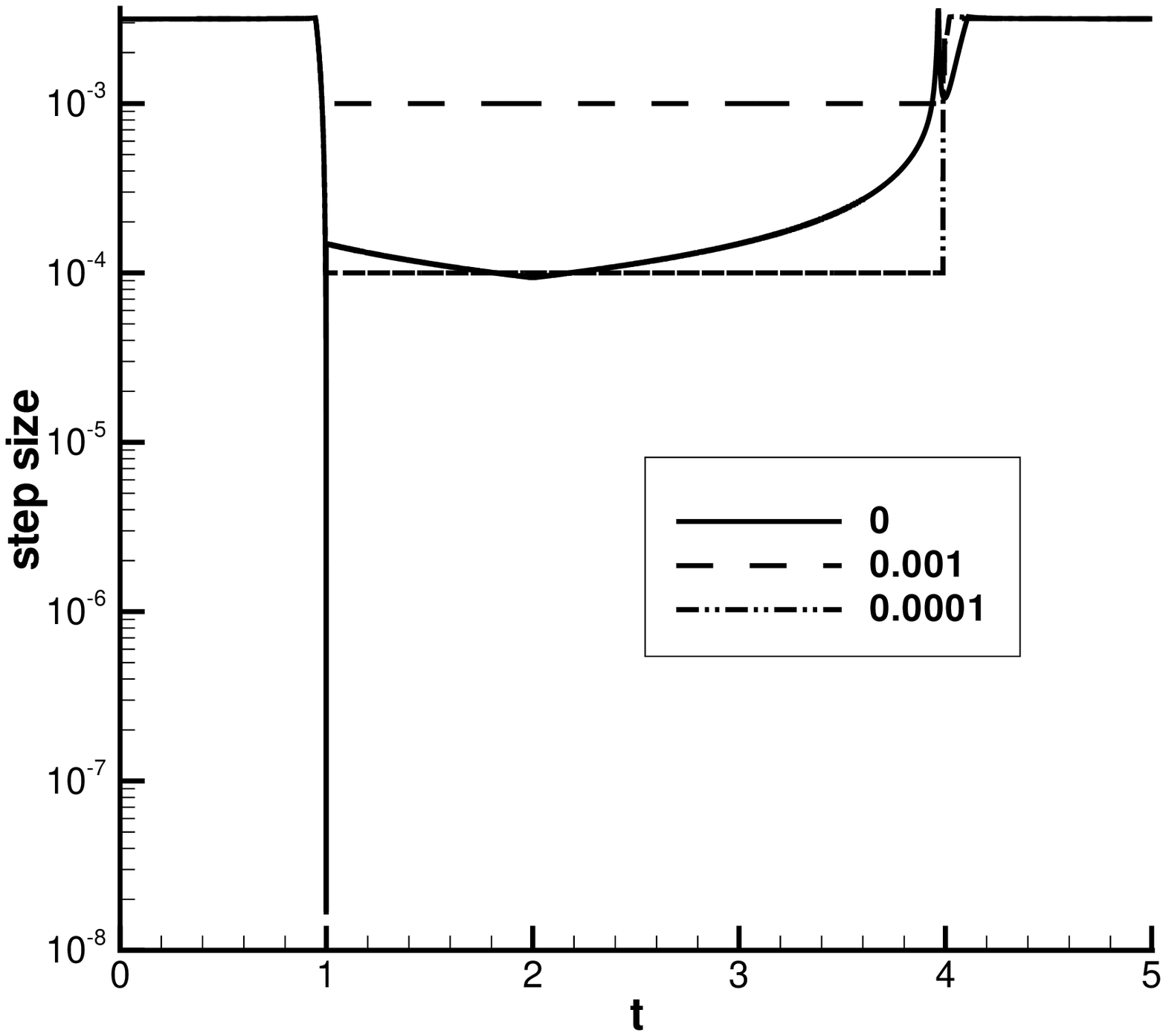}
 }
 \caption{Results obtained by Algorithm 3 with $tol=10^{-5}$ and different values of $\Delta t_{\min}$ specified in the legend. 
 Left: evolution of $q(t^k)$ as compared to the exact solution.
 Right: evolution of $\Delta t_k$.
 }
 \label{algo3}
\end{figure}

Let us now turn to another series of numerical experiments taking $\varepsilon=10^{-3}$.  The results are reported in Figure \ref{eps0001}. 
Algorithm 1 is now unable to produce a physically acceptable solution even with very small $\Delta t = 0.0001$. 
Algorithm 2, on the other hand, works fine. 
In order to observe better the quasi-contact region and the effect of introducing the threshold, we zoom on small distances by passing to a log-scale in Figure 
\ref{eps0001log} (left).  We report also some numerical experiments aiming at the determination of an optimal value of the threshold $q_s$. As said already, our strategy
is to choose it so that $n(q_s)=1/(C\Delta t)$ taking $C=20$ in all the preceding experiments. Figure \ref{eps0001log} (left) illustrates the results obtained fixing 
$\Delta t=0.001$ and taking three different values for $q_s$ by varying $C$. These results confirm two general observations: taking $q_s$ too small may 
deteriorate the accuracy of the solution (indeed, the original goal of the introduction of $q_s$ was to avoid the extremely small values of $q$). On the other hand,
taking $q_s$ too large may unnecessarily perturb the solution in the regions where a straight discretization could give more precise results. Thus, there should be an optimal
choice of $q_s$ which seems to be not too far from the formula mentioned above with $C=20$.

We conclude by a series of experiments using Algorithm 3 with adapted stepsize. We take again the same governing equation with $\varepsilon=10^{-3}$ and run Algorithm 3 
setting  the tolerance to $tol=10^{-5}$ and taking a number of values for $\Delta t_{\min}$. The results are reported in Figure \ref{algo3} with the evolution of $q(t^k)$ 
on the left and that of $\Delta t_k$ on the right. We give in particular the results with $\Delta t_{\min}=0$ so that Algorithm 3 is reduced to a standard time marching scheme
with automatically adapted step sizes. It is interesting to note that step sizes $\Delta t$ of the order $10^{-4}$ are actually sufficient to satisfy our error tolerance criterion almost 
everywhere apart from a very small range of time around 1 where the Algorithm chooses $\Delta t$ of the order $10^{-8}$. Setting  $\Delta t_{\min}$ to $10^{-3}$ or $10^{-4}$ avoid such small
time steps and gives fairly good results.

\section*{Acknowledgements.} We would like to thank the organizers of the conference ECM'09 (International Conference of Microfluidics and Complex flow) Lassaad El Asmi and Mourad Ismail 
for giving us the opportunity to present this work. We had a number of useful discussions during this conference, especially with Aline Lefebvre-Lepot and Fran\c cois Feuillebois, 
to whom we are very grateful.

\appendix
\section{An estimate for the drag (lubrication) force}

This part is devoted to computing the first order expansion of the viscous drag exerted by a fluid 
on a disk approaching a plane wall. Let $\Omega\subset\mathbb R^2$ be the half plane $\Omega= \{(x,y)  \text{ with } y>0\}$, 
the boundary of $\Omega$ is $\partial\Omega=\{(x,0),\ x\in\mathbb{R}\}$ (see Figure \ref{fig_Stokes}). We assume that the fluid fills the domain $\mathcal F_q := \Omega \setminus \overline{\mathcal B_q}$
where $\mathcal B_q := B((0,R+q),R)$ is the domain occupied by the solid disk of radius $R>0$ and placed at the distance $q >0$ from  the boundary is $\partial\Omega$. 
Our notations are lightly different here from that of Section \ref{general}. Indeed, domains $\mathcal F_q$ and $\mathcal B_q$ are now indexed by $q$ rather than $t$ since 
the eventual dependence of $q$ on time does not interest us in this Appendix. Assume moreover that the disk has a prescribed velocity $\mathbf{V}=V\mathbf{e}_2$ and 
zero angular velocity. The motion of the fluid is governed by Stokes equations \eqref{stokes} with $\mathbf{g}=0$ and zero boundary conditions at the infinity for $\mathbf{u}$. 
It is classical that (\ref{stokes}) is well-posed in a suitable framework
and has a unique solution $(\mathbf u,p)$ (the pressure $p$ is defined up to a constant).
In particular, the viscous drag force $\mathbf{F}$ exerted by a fluid flow $(\mathbf u,p)$ on $\mathcal B_q$ and computed by \eqref{eq_drag} is well-defined as a function of $V$ and $q.$ 
Our aim is to prove the following result:
\begin{theorem} \label{thm_Sigma}
The viscous drag force satisfies $\mathbf{F}=-n(q)\mathbf{V}$ where
$$
n(q) =  3 \sqrt{2} \pi \nu \left( \frac{R}{q} \right)^{\frac{3}{2}} \left[1 + \varepsilon(q)\right], 
$$
with  $\varepsilon(q) \to 0$ when $q \to 0.$
\end{theorem}

\begin{remark}
 The equivalent result in the three-dimensional setting is well-known but it is new in two dimensions.
Whatever the novelty of this result, the main contribution in this section is that we provide a new way for computing the first order expansion of the viscous drag.
This new method is more robust than the known computations \cite{Cooley&Oneill69,Dean&Oneill64,ONeill64,ONeill&Stewartson67}. In particular, we claim it extends to the three-dimensional
setting with small changes and can be adapted to other boundary conditions. As an illustration,
our method is a well-designed tool for estimating the distance between solid bodies in solutions to the fluid-structure interaction system with the full newtonian Navier Stokes 
equation on the fluid domain
\cite{DGVHillairet,Hillairet07,HillairetTakahashi09}. 
\end{remark}

The proof of Theorem \ref{thm_Sigma} is divided into three steps. 
First, we recall the variational formulation for the solution $(\mathbf u,p)$ to  (\ref{stokes})
and apply it to compute the associated drag coefficient $n(q)$ in the expression for viscous drag $\mathbf{F}=-n(q)\mathbf{V}$. 
Secondly, we deduce from the variational formulation the lower and upper bounds for $n(q)$. 
We conclude the proof by computing an asymptotic expansion of the bounds of this range.

\paragraph{1. The variational formulation for the Stokes system (\ref{stokes}).}
As $(\mathbf u,p)$ in (\ref{stokes}) and the formula \eqref{eq_drag} for the drag depend linearly on $\mathbf{V}$, we set $\mathbf{V} = \mathbf{e}_2$ in what follows. 
It is classical to extend the velocity $\mathbf{u}$ to the whole domain $\Omega$ by setting it to $\mathbf{e}_2$ inside $\mathcal{B}_q$ and 
to look for $\mathbf{u}$ in the following space 
$$
 D^{1,2}_0(\Omega) := \Big\{ \mathbf{w} \in L^1_{loc}(\Omega) \text{ with } \nabla \mathbf  w \in L^2(\Omega) , \text{ } \text{div} \ \mathbf w =0 \text{ in }\Omega
 \text{ and } \mathbf w = 0 \text{ on $\partial\Omega$} \Big\}.
$$
The solution of (\ref{stokes}) is associated with the minimization of Dirichlet integral $\mathbb{D}(\mathbf w) = \int_{\Omega} |\nabla(\mathbf w)(x,y)|^2 \text{d$x$d$y$}$ 
over the subset $Y_q$ of $D^{1,2}_0(\Omega) $:
$$
Y_q = \Big \{ \mathbf w \in D^{1,2}_0(\Omega) \text{ such that } \mathbf w_{|_{\mathcal B_q}} = \mathbf e_2 \Big\}. 
$$
Let us denote this minimum by $E_D(q)$, { {\it \textit{i.e.}}}
\begin{equation}
\label{eq_minim}
E_D(q) := \min \left\{  \mathbb{D}(\mathbf w) \ ; \ \mathbf{w} \in Y_q \right\},
\end{equation}
As the Dirichlet integral $\mathbb D$ is a strictly convex functional on $Y_q$ it has a unique minimizer $\mathbf u_q$  for any $q>0.$ It is easy to see that 
$\mathbf u_q$ gives the solution of \eqref{stokes}. Indeed, we have for any $\mathbf w \in \mathcal{C}^{\infty}_c(\mathcal F_q)$ 
such that $\text{div}\ \mathbf w =0,$
$$
 \int_{\Omega}  \nabla\mathbf u(x,y) : \nabla\mathbf w(x,y)\text{d$x$d$y$} = 0,
$$
which is the variational formulation of \eqref{stokes}. The pressure ${ {p_q}} \in L^2_{loc}(\Omega)$ can be then recovered as the Lagrange multiplier corresponding
to the constraint $\text{div}\ \mathbf u =0.$ We notice that ${ {p_q}}$ is unique up to a constant and the ellipticity of the Stokes problem implies that 
$(\mathbf u_q,{ {p_q}})$ is smooth on $\overline{\mathcal F}_q$ so that it  furnishes a classical solution to (\ref{stokes}). 
We refer to \cite{Galdi} for more details and also for a proof of the converse implication.

Given $\mathbf w \in \mathcal C^{\infty}_c(\Omega)$ such that $\text{div} \ \mathbf w = 0$ and $\mathbf w = \mathbf e_2$  on  $\mathcal B_q,$ a straightforward integration by parts yields
(recall that $\mathbf{n}$ is the unit normal looking into $\mathcal F_q$):
\begin{align*}
\nu\int_{\Omega} \nabla { {\mathbf u_q}}(x,y) : \nabla \mathbf w(x,y)\text{d$x$d$y$} &=
2\nu \int_{\mathcal F_q} D({ {\mathbf u_q}})(x,y) : \nabla\mathbf w(x,y)\text{d$x$d$y$}  \\ 
& = -\int_{\partial \mathcal B_q} [ 2\nu  D({ {\mathbf u_q}}) - { {p_q}} I_2]  \mathbf n \text{d$\sigma$} \cdot \mathbf e_2.
\end{align*}
Taking a suitable family of approximation of $\mathbf u_q$ (see \cite{Galdi} for details) we obtain in the limit:
$$
\nu E_D(q) = -\int_{\partial \mathcal B_q} [ 2\nu  D({ {\mathbf u_q}}) - { {p_q}} I_2]  \mathbf n \text{d$\sigma$} \cdot \mathbf e_2.
$$
The drag $\mathbf{F}$ is parallel to the velocity $\mathbf{V}$ for symmetry reasons, { {\it \textit{i.e.}}} $\mathbf{F}=(\mathbf{F}\cdot\mathbf{e}_2)\mathbf{e}_2$. 
We have thus the following result.
\begin{lemma}
\label{lemme_minim}
For any given $q>0$ and $\mathbf{V}$ parallel to $\mathbf{e}_2$, the drag force is given by $\mathbf{F}=-n(q)\mathbf{V}$ with $n(q)=\nu E_D(q).$ 
\end{lemma}

\paragraph{2. Upper and lower bounds for $E_D(q)$.}
We apply the above variational formulation to bound $E_D(q).$
To this end, given $q>0,$ we first prove that the Dirichlet integral of any $\mathbf u  \in Y_q$ is greater than
some $m_D(q)$ depending only on $q.$ This gives us a lower bound for $E_D(q).$
Then, we construct a suitable ${ {\tilde{\mathbf u}_q}} \in Y_q$ and compute its
Dirichlet integral $M_D(q)$. This gives an upper bound for $E_D(q).$ We finally compare
$m_D$ and $M_D$ for small values of $q$ to prove Theorem \ref{thm_Sigma}.

As the scale invariance of our problem implies 
that $n(\cdot)$ is actually a function of $q/R$, it is sufficient to consider the case $R=1$, which we admit in what follows. 
For any $x \in (-1,1),$ $y < 1+q$ such that $(x,y) \in \partial \mathcal B_q$
there holds $y = 1 + q - \sqrt{1-x^2} := \delta_q(x)$ and we denote 
$$
\Omega_q := \{(x,y) \in \mathbb R^2 \text{ such that } x \in (-1/2,1/2) \text{ and } y \in (0,\delta_q(x)) \}.
$$

Given a smooth $\mathbf u  \in Y_q,$ we introduce 
$$
\psi(x,y) = -\int_{0}^y u_1(x,z) \text{d$z$}.
$$
Then, straightforward computations imply that $\psi \in \mathcal{C}^{\infty}({ {\overline{\mathcal{F}}_q}})$ and  satisfy
$\mathbf u = \nabla^{\bot} \psi=(-\partial_y\psi, \partial_x\psi),$ so that we have the boundary conditions:
\begin{equation} \label{eq_BCpsi}
\left\{
\begin{array}{rclrcll}
\psi(x,y) &=& 0, & \partial_y \psi(x,y) = 0, & \text{ on $\partial\Omega,$} \\
\partial_x \psi(x,y) &=& 1, & \partial_y \psi(x,y) = 0, &\text{ on $\partial \mathcal B_q$}.    
\end{array}
\right.
\end{equation}
On the other hand, we compute the Dirichlet integral $\mathbb D(\mathbf u)$ with respect to $\psi.$
This yields:
$$
\mathbb D(\mathbf u) \geq \int_{\Omega_q} |\partial_{yy} \psi(x,y)|^2 \text{d$x$d$y$}. 
$$
We denote $I(\psi)$ the integral on the right-hand side of the above inequality. Thus, $\mathbb D(\mathbf{u})$ is greater than the minimum of $I(\psi)$ over smooth $\psi$ satisfying the boundary conditions \eqref{eq_BCpsi}.
This minimum is computed in the following lemma:
\begin{lemma}
Given $q >0$ and $\psi \in \mathcal{C}^{\infty}(\overline{\Omega}_q),$ satisfying boundary conditions \eqref{eq_BCpsi}
there holds $I(\psi) \geq I(\psi_q)$ where :
$$
\psi_q(x,y) = x \left[ \dfrac{y}{\delta_q(x)}\right]^{2} \left( 3 - 2 \left[ \dfrac{y}{\delta_q(x)}\right] \right),
\quad \forall \, (x,y) \in \overline{\Omega}_q.
$$
\end{lemma}
\begin{proof}
Given any $\psi(x,y) \in \mathcal{C}^{\infty}(\overline{\Omega}_q),$ satisfying  \eqref{eq_BCpsi}, boundary conditions on $\partial \mathcal B_q$
imply $\psi(x,\delta_q(x)) = C + x,$ $\forall x \in (-1/2,1/2)$, with some $C \in \mathbb R$. Thus $\psi$ satisfies $\forall x \in (-1/2,1/2)$
$$
\psi(x,0) = 0, \quad \psi(x,\delta_q(x)) = C+x ,\quad \frac{\partial\psi}{\partial y}(x,0) =0, \quad \frac{\partial\psi}{\partial y}(x,\delta_q(x)) = 0.
$$ 
Then, for arbitrary $x \in (-1/2,1/2),$ excepting eventually $x=-C$, we introduce $\chi_x(t) = \psi(x,t\delta_q(x))/(x+C)$, $t\in(0,1)$, which satisfies 
\begin{equation}\label{bchi}
\chi_x(0) = 0, \quad \chi_x(1) = 1 ,\quad \chi_x'(0) =0, \quad \chi_x'(1) = 0.
\end{equation}
By a standard optimization argument we see that the minimum of $\int_{0}^{1} |\chi_x''(t)|^2 dt$ over all smooth $\chi_x(t)$ satisfying \eqref{bchi} is attained on the function $\eta(t)$ 
such that $\eta^{(4)}(t)=0$ on $(0,1)$, which is given by $\eta(t)=t^2(3-2t)$.
This yields
\begin{align}
\int_{\Omega_q} |\partial_{yy} \psi(x,y)|^2 \text{d$x$d$y$} &= \int_{-1/2}^{1/2}\int_{0}^{1}   \frac{(x+C)^2}{|\delta_q(x)|^3} |\chi_x''(t)|^2 \text{d$t$d$x$} \notag\\
&\geq \int_{-1/2}^{1/2}\int_{0}^{1}   \frac{(x+C)^2}{|\delta_q(x)|^3} |\eta''(t)|^2 \text{d$t$d$x$}. \label{intmD}
\end{align}
Minimizing the last integral with respect to $C,$ we obtain $C=0$. The last inequality becomes equality if we take $\psi(x,y)=\psi_q(x,y)=x\eta(y/\delta_q(x))$.
\end{proof}

Finally, a lower bound is given by 
$$
m_D(q) = \int_{\Omega_q} |\partial_{yy} \psi_q(x,y)|^2 \text{d$x$d$y$}.
$$
On the other hand, it is always possible to extend $\psi_q$ on the whole $\Omega_q$ to some 
$\tilde{\psi}_q \in \mathcal{C}^{\infty}(\mathcal F_q) \cap \mathcal{C}^{\infty}(\mathcal B_q) \cap \mathcal{C}(\Omega)$ 
and such that ${ {\tilde{\mathbf{u}}_q := \nabla \tilde{\psi}_q}} \in Y_q.$ For instance one might interpolate (in the $x$-variable)
$\psi_q$ with a suitable truncation of $\psi_0(x,y) = x.$  
As, outside $\Omega_q,$ the solid disk remains at a positive distance of $\partial\Omega,$ the truncation and interpolation
function can be made independent of $q.$ In particular there holds:
$$
\int_{\Omega} |{ {\nabla \tilde{\mathbf u}_q}}(x,y)|^2 \text{d$x$d$y$} = m_D(q) + r(q) + \int_{\Omega_q}\left[  |\partial_{xx} \psi_q(x,y)|^2 + |\partial_{xy} \psi_q(x,y)|^2 \right] \text{d$x$d$y$} := M_D(q) ,
$$
with $r$ a bounded function of $q$. 

\paragraph{3. Asymptotic expansion of $m_D(q)$ and $M_D(q)$.}
\begin{lemma}
There holds, for small values of $q$ : 
\begin{equation}\label{mDass}
m_D(q) = 3 \sqrt{2} \pi \dfrac{1 + \varepsilon(q)}{q^{3/2}},
\end{equation}
with $\varepsilon(q) \to 0$ when $q \to 0.$
\end{lemma}
\begin{proof}
As already seen in the proof of the preceding Lemma, $m_D(q)$ is given by the integral \eqref{intmD} with $C=0$ and $\eta(t)=t^2(3-2t)$. 
Integrating with respect to $t$ and performing a change of variables $x\to\sqrt{q}x$ 
yields
$$
m_D(q) = 12 \displaystyle{\int_{-1/2}^{1/2}  \dfrac{x^2}{|\delta_q(x)|^3} \text{d$x$}}
= \dfrac{12}{q^{3/2}}  \displaystyle{\int_{-1/(2\sqrt{q})}^{1/(2\sqrt{q})}} \dfrac{ x^2 \text{d$x$}}{(\delta_q(\sqrt{q}x)/q)^3}
$$
Expanding to the first order, we have:
$$
\lim_{q \to 0}\dfrac{ x^2 }{(\delta_q(\sqrt{q}x)/q)^3} =
\lim_{q \to 0} \dfrac{x^2}{[(1+q - \sqrt{1- qx^2})/q]^3} = \dfrac{x^2}{(1+x^2/2)^3}.
$$
We remark moreover that standard geometric arguments imply that $\delta_q(x) \geq q + x^2/2$ or $\delta_q(\sqrt{q}x)/q \geq 1 + x^2/2$
so that we can apply the Lebesgue theorem to obtain:
$$
\lim_{q \to 0} \displaystyle{\int_{-1/(2\sqrt{q})}^{1/(2\sqrt{q})}} \dfrac{ x^2 \text{d$x$}}{(\delta_q(\sqrt{q}x)/q)^3} = \int_{-\infty}^{\infty} \dfrac{x^2 \text{d$x$}}{(1+x^2/2)^{3}} = \dfrac{\pi}{2\sqrt{2}}.
$$
\end{proof}

One can also check that $\partial_{yy}\psi_q$ is the second derivative of $\psi_q$ which diverges the fastest when
$q$ goes to $0.$ Thus,  $M_D(q)-m_D(q)=o(m_D(q))$ as $q\to 0$ so that \eqref{mDass} holds also for $M_D(q)$, which proves Theorem \ref{thm_Sigma}.

\begin{remark}\label{rem3D}
The same method can be easily adapted to the three dimensional case. More precisely, the result of Lemma \ref{lemme_minim} is still valid : the viscous drag exerted by a fluid on a sphere can be computed as the result of a minimization problem of the same type as the one given by equation (\ref{eq_minim}). Next, since the problem is axisymmetric, one can use cylindrical coordinates and perform the same asymptotical analysis as in the 2D case, in order to obtain that
$$
n(q) =  6\pi V \nu \frac{R^2}{q} \left[1 + \varepsilon(q)\right], 
$$
with  $\varepsilon(q) \to 0$ when $q \to 0.$

\end{remark}

\subsection*{Numerical evaluation of the drag force.}

In order to compare the theoretical estimate given by Proposition \ref{thm_Sigma} to a numerical evaluation of the force $\mathbf{F}$, we carry out the following simulations: we begin with solving the Stokes problem in a fluid with viscosity $\nu=1$ and then compute the corresponding force exerted by the fluid on a particle of radius $R=0.1$ situated above a plane at a distance $q$. 
The computations are performed by using the Finite-Element solver FreeFem++ \cite{freefem} and the results are plotted in Figure \ref{fig_Force} (dashed line with squares). 
They agree with the asymptotic expansion $F\sim 3\sqrt{2}\pi V\nu \left(\frac{R}{q}\right)^{\frac{3}{2}}$ (solid line).

\begin{figure}[h!]
\centerline{
     \includegraphics[scale=0.35,angle=-90]{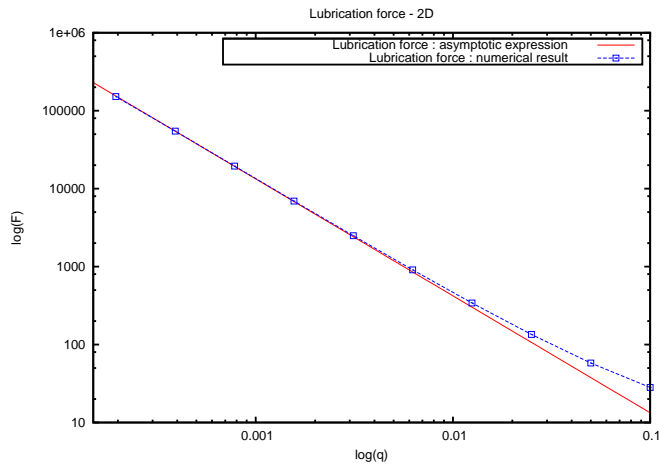}
}
     \caption{Approximation of the lubrication force}
     \label{fig_Force}
\end{figure}

\section{Non-dimensional form of the governing equations (\ref{odem}) involving the lubrication force}
Let us first consider the 2D case. { {Eliminating}} the velocity $v$ from (\ref{odem}) and specifying $m=\rho_s\pi R^2$ with $\rho_s$ the density of the solid, we arrive at 
$$
\rho_s\pi R^2\ddot{q}=-3\sqrt{2}\pi \nu \left( \frac{R}{q}\right)^{\frac{3}{2}}\dot{q} + (\rho_s-\rho_f)\pi R^2 g.
$$
In order to understand the typical values of parameters in this ODE, we should pass to non-dimensional variables,
which can be introduced as follows:
$$
q=Rq',\ t=Tt',\ g=g_{char}g'
$$
where the radius $R$ is used as the length scale, the time scale is denoted by $T$ and $g_{char}$ is a typical value of the external force density, so that
{ {$g_{char}\sim 10 m/s^2$}}. We choose then the time scale $T$ so that the non-dimensional external force becomes of order 1,
{ {\it \textit{i.e.}}} $\frac{T^2}{R}\frac{\rho_s-\rho_f}{\rho_s} g_{char} =1$ and obtain the non-dimensional ODE (dropping the primes) \eqref{ode2D}.

The asymptotic in the 3D case is is given by $n(q)=6\pi\nu\frac {R^2}q$ (cf. Remark \ref{rem3D}) so that starting from the dimensional ODE
$$
\rho_s\frac{4}{3}\pi R^3\ddot{q}=-6\pi \nu R^2 \frac{\dot q}{q} + (\rho_s-\rho_f)\frac{4}{3}\pi R^3 g \qquad \textrm{(3D).}
$$
and using the similar non-dimensionalizations as before we obtain \eqref{ode3D}.

\bibliographystyle{plain}
\bibliography{biblio}
\end{document}